\documentclass{article}
\usepackage{amsmath,amssymb,latexsym,amsthm,mathrsfs,appendix}

\newtheorem{theo}{Theorem}[section]
\newtheorem{pro}[theo]{Proposition}
\newtheorem{lem}[theo]{Lemma}
\newtheorem{cor}[theo]{Corollary}

\newcommand{\ra}{\rightarrow}
\theoremstyle{definition}
\newtheorem{defin}[theo]{Definition}

\newtheorem{exa}{Example}[section]
\newtheorem{prob}[theo]{Problem}

\theoremstyle{remark}
\newtheorem{rem}[theo]{Remark}

\begin{document}

\title{Random walks and isoperimetric profiles under moment conditions}
\author{Laurent Saloff-Coste\thanks{%
Partially supported by NSF grant DMS 1004771 and DMS 1404435} \\
{\small Department of Mathematics}\\
{\small Cornell University} \and Tianyi Zheng \\
{\small Department of Mathematics}\\
{\small Stanford University} }
\maketitle

\begin{abstract}
Let $G$ be a finitely generated group equipped 
with a finite symmetric generating set and the associated word 
length function $|\cdot |$. We study the behavior 
of the probability of return 
for random walks driven by symmetric measures $\mu$ that are such that
$\sum \rho(|x|)\mu(x)<\infty$ for increasing regularly varying or 
slowly varying functions $\rho$, for instance, $s\mapsto (1+s)^\alpha$,
$\alpha\in (0,2]$, or 
$s\mapsto (1+\log (1+s))^\epsilon$, $\epsilon>0$.
For this purpose we develop new relations between the isoperimetric profiles
associated with different symmetric probability measures. These techniques 
allow us to obtain a sharp $L^2$-version of Erschler's inequality 
concerning the F\o lner functions of wreath products. Examples and 
assorted applications are included.
\end{abstract}

\section{Introduction}
\setcounter{equation}{0}

Let $G$ be a finitely generated group. The following notation will be
used throughout this work.
Let $S=(s_1,\dots,s_k)$ be a fixed generating 
$k$-tuple and $S^*=\{e,s_1^{\pm1},\cdots, s_k^{\pm 1}\}$ 
be the associated symmetric generating set. 
Let $|\cdot|$ be the associated word-length 
so that $|g|$ is the least $m$ such that $g=\sigma_1\dots\sigma_m$ with 
$\sigma_i \in \mathcal S^*$ (and the convention that $|e|=0$, where $e$ 
is the identity element in $G$). Let $B(r)=\{g\in G:|g|\le r\}$ and 
let $V$ be the associated volume growth 
function defined by 
$$V(r)=|\{g\in G: |g|\le r\}|$$
where $|\Omega|=\#\Omega$ is the number of elements in $\Omega\subset G$. 
For $r\ge 1$, let
$\mathbf u_r$ be the uniform probability measure on $B(r)$ and set 
$\mathbf u=\mathbf u_1$, 
that is, 
\begin{equation}\label{nu}
\mathbf u_r=\frac{1}{|B(r)|}\mathbf 1_{B(r)}\mbox{ and } 
\mathbf u=\mathbf u_1=\frac{1}{|S^*|}\mathbf 1_{S^*}.
\end{equation}

Given two functions $f_1,f_2$ taking real values but defined on an arbitrary 
domain (not necessarily a subset of $\mathbb R$), we write  $f\asymp g$
to signify that there are constants $c_1,c_2\in (0,\infty)$ such that  
$c_1f_1\le f_2\le c_2f_1$.
Given two monotone real functions $f_1,f_2$, write $f_1\simeq f_2$ 
if there exists $c_i\in (0,\infty)$ such that 
$c_1f_1(c_2t)\le f_2(t)\le c_3f_1(c_4t)$ on the domain of definition of 
$f_1,f_2$ (usually, $f_1,f_2$ will be defined on a neighborhood of $0$ or 
infinity and tend to $0$ or infinity at either $0$ or infinity. In some cases, 
one or both functions are defined only on a countable set such as $\mathbb N$). 
We denote the associated order by $\lesssim$.
Note that the equivalence relation $\simeq$ distinguishes between power 
functions of different degrees and between stretched exponentials $\exp(-t^\alpha)$ 
of different exponent $\alpha>0$ but does not distinguishes between 
different rates of exponential growth or decay (e.g., $2^n\simeq 5^n$).
It is not hard to verify that the volume growth functions associated with 
two finite symmetric generating sets of a given group $G$ 
are $\simeq$-equivalent.

Given an arbitrary probability measure $\phi$ on a group $G$, 
we let $(S_n)_0^\infty$ 
denote the trajectory of the random walk driven by $\phi$ 
(often started at the identity element $e$).  We let $\mathbf P_\phi$
be the associated measure on $G^\mathbf N$ with $S_0=e$  and 
$\mathbf E _\phi$ the corresponding expectation 
$\mathbf E^x_\phi(F)=\int_{G^{\mathbf N}}
F(\omega)d\mathbf P^x_\phi(\omega)$. In particular,
$$\mathbf P_\phi(S_n=x)=\mathbf E_\phi(\mathbf 1_{x}(S(n)))=\phi^{(n)}(x).$$

\subsection{The random walk invariants $\Phi_{G,\rho}$ and 
$\widetilde{\Phi}_{G,\rho}$}
In \cite{PSCstab}, it is proved that, for any finitely generated group $G$, 
there exists a function 
$\Phi_G:\mathbb N\ra (0,\infty)$ such that, 
if $\mu$ is a symmetric probability measure with generating 
support and finite second moment, that is $\sum |g|^2\mu(g)<\infty$, then
$$\mu^{(2n)}(e)\simeq \Phi_G(n).$$
Further, \cite{PSCstab} proves that $\Phi_G$ is an invariant of quasi-isometry.
Throughout this paper and referring to definition  (\ref{nu}), 
we will use $n\mapsto \mathbf u^{(n)}(e)$ as our favorite representative 
for $\Phi_G$.

In \cite{BSClmrw}, A. Bendikov and the first author considered the question of
finding lower bounds for the probability of return $\mu^{(2n)}(e)$ when $\mu$
is symmetric and is only known to have a finite moment of 
some given exponent lower than $2$. 
Very generally, let $\rho: [0,\infty)\ra [1,\infty)$ be  given.
We say that a measure $\mu$ has finite $\rho$-moment if 
$\sum \rho(|g|)\mu(g)<\infty$. We say that $\mu$ has finite weak-$\rho$-moment
if 
\begin{equation} \label{weak}
W(\rho,\mu)
:= \sup_{s>0} \left\{s \mu(\{g: \rho(|g|) >s\})\right\} <\infty. 
\end{equation}

\begin{defin}[Fastest decay under $\rho$-moment] \label{def1}
Let $G$ be a countable group.
Fix a function $\rho:[0,\infty)\ra [1,\infty)$.
Let $\mathcal S_{G,\rho}$ be the set of all symmetric
probability  $\phi$ on $G$ with the properties that
$\sum \rho(|g|)\phi(g) \le 2\rho(0)$. Set
$$\Phi_{G,\rho}: n\mapsto \Phi_{G,\rho}(n)= \inf\left\{
\phi^{(2n)}(e) : \phi \in \mathcal S_{G,\rho}\right\}. $$
\end{defin}
In words, $\Phi_{G,\rho}$ provides the best lower bound valid 
for all convolution powers of probability measures in 
$\mathcal S _{G,\rho}$. The following variant deals with finite weak-moments
and will be key for our purpose.

\begin{defin}[Fastest decay under weak-$\rho$-moment] \label{def2}
Let $G$ be a countable group.
Fix a function $\rho:[0,\infty)\ra [1,\infty)$.
Let $\widetilde{\mathcal S}_{G,\rho}$ be the set of all symmetric
probability   $\phi$ on $G$ with the properties that
$W(\rho,\mu)\le 2\rho(0)$. Set
$$\widetilde{\Phi}_{G,\rho}: n\mapsto \widetilde{\Phi}_{G,\rho}(n)= 
\inf\left\{
\phi^{(2n)}(e) : \phi \in \widetilde{\mathcal S}_{G,\rho}\right\}. $$
\end{defin}

\begin{rem} Assume that $\rho$ has the property that 
$\rho(x+y)\le C( \rho(x)+\rho(y))$. Under this natural condition 
\cite[Cor 2.3]{BSClmrw} shows that $\Phi_{G,\rho} $ and 
$\widetilde{\Phi}_{G,\rho}$ stay strictly positive. Further,
\cite[Prop 2.4]{BSClmrw} shows that, for any symmetric probability measure 
$\mu$ on $G$ such that $\sum \rho(|g|)\mu(g)<\infty$
(resp. $W(\rho,\mu)<\infty$), 
there exist constants $c_1,c_2$ (depending on $\mu$)
 such that
$$\mu^{(2n)}(e)\ge c_1 \Phi_{G,\rho}(c_2 n)$$
 (resp. $\mu^{(2n)}(e)\ge c_1 \widetilde{\Phi}_{G,\rho}(c_2n)$).
This makes it very natural to consider not the functions $\Phi_{G,\rho}$ 
and $\widetilde{\Phi}_{G,\rho}$ themselves but their equivalence classes under 
the equivalence relation $\simeq$.
\end{rem}
\begin{rem}The reader may wonder why we are considering the weak-moment variant 
$\widetilde{\Phi}_{G,\rho}$. The reason is that it appears 
to be the more natural version of the two variants. 
For instance, when $G=\mathbb Z$, we do not know what the behavior  
of
$\Phi_{\mathbb Z,\rho_\alpha}$ is for $\alpha\in (0,2)$ whereas it is 
very well known and easy to show that 
$\widetilde{\Phi}_{\mathbb Z,\rho_\alpha}(n)\simeq n^{-1/\alpha}$, $\alpha\in (0,2)$.
\end{rem}

When $\rho_\alpha(s)= (1+s)^\alpha$ with $\alpha>0$, 
the main result of \cite{PSCstab}
implies that 
$$\Phi_{G}\simeq \Phi_{G,\rho_2}\simeq 
\Phi_{G,\rho_\alpha}\simeq \widetilde{\Phi}_{G,\rho_\alpha} 
\mbox{ for all } \alpha>2.$$
This holds even so there are great many finitely generated groups $G$
(indeed, uncountably many), for which we do not know how $\Phi_G$ behaves. 
For a group $G$ of polynomial volume growth with $V(r)\simeq r^D$, we know that
(see \cite{HSC,BSClmrw,SCZ0} and the references therein)
$$\Phi_{G}(n)\simeq n^{-D/2},\;\;\; 
\widetilde{\Phi}_{G,\rho_\alpha}(n)\simeq n^{-D/\alpha}, \alpha\in (0,2)$$ 
showing that the condition $\alpha>2$ cannot be relaxed.

Definitions \ref{def1}-\ref{def2} lead to the following problems.

\begin{prob} \label{prob1}
Let $G,H$ be two finitely generated groups. Let $\rho,\theta$
be two (nice) increasing functions with $\theta\le \rho$.   
\begin{enumerate} 
\item Does $\Phi_{G}\simeq \Phi_H$ imply 
$\widetilde{\Phi}_{G,\rho}\simeq \widetilde{\Phi}_{H,\rho}$?
\item Does $\widetilde{\Phi}_{G,\rho}\simeq \widetilde{\Phi}_{H,\rho}$ imply 
$\widetilde{\Phi}_{G,\theta}\simeq \widetilde{\Phi}_{H,\theta}$? 
\item Fix $\alpha\in (0,2)$. Is it true that 
$\widetilde{\Phi}_{G,\rho_\alpha}\simeq \widetilde{\Phi}_{H,\rho_\alpha}$ 
is equivalent to $\Phi_{G}\simeq \Phi_{H}$?
\item What is the behavior of $\widetilde{\Phi}_{G,\rho_2}$ and, 
more generally, of  $\widetilde{\Phi}_{G,\rho}$ when $\rho$ 
is close to $t\mapsto t^2$?
\end{enumerate}
\end{prob}

In contemplating these questions, it is reasonable to make additional 
assumptions on the functions $\rho,\theta$, for instance, one may want to 
assume that
 $\rho,\theta$ are continuous increasing functions satisfying the 
doubling condition $\exists \,C>0,\;\;\forall \,t>0,\;f(2t)\le Cf(t)$. 
Or one may even want to assume that $\rho,\theta$ are taken from a given 
list of  functions such a
$$s\mapsto (1+s)^{\alpha_0} \prod_{k=1}^m \log^{\alpha_k}_{[k]}(s)$$ 
with $\alpha_0\ge 0$ and $\alpha_1,\dots,\alpha_m\in \mathbb R$ (the first 
non-zero $\alpha_i$ should be positive). 
Here and in the rest of this paper, $\log_{[k]}$ is defined inductively by
$\log_{[k]}(s)=1+\log(\log_{[k-1]}(s))$, $\log_{[1]}(s)=1+\log(1+s)$. 
For instance, an interesting restricted version of the first question is 
concerned with the case when $\rho\in\{\rho_\alpha: \alpha\in (0,2)\}$.

If $G$ has polynomial volume 
growth of degree $D>0$ then
$\widetilde{\Phi}_{G,\rho_\alpha}(n)\simeq n^{-D/\alpha}$, $\alpha\in (0,2)$ 
while  (\cite{SCZ0,SCZ-pollog})
$$\widetilde{\Phi}_{G,\log_{[1]}^\epsilon} (n) 
\simeq \exp(-n^{1/(1+\epsilon)}).$$   
So,
$\widetilde{\Phi}_{G,\alpha}$ distinguishes between different degrees of 
growth whereas 
$\widetilde{\Phi}_{G,\log_{[1]}^\epsilon}$ does not (except between $D=0$ and $D>0$).

From a heuristic point of view, there are reasons to believe 
that the slower the function $\rho$ grows, the coarser the group invariant
$\widetilde{\Phi}_{G,\rho}$ is (modulo the equivalence relation $\simeq$). 
The first two questions stated in Problem \ref{prob1} relate to
this heuristic and ask if this conjectural picture is correct. Namely, 
if $\theta\le \rho$, is it correct that the partition one obtains by 
considering the classes of groups  sharing the same 
$\widetilde{\Phi}_{G,\theta}$ are 
obtained by lumping together classes corresponding to 
$\widetilde{\Phi}_{G,\rho}$.   The third question in Problem \ref{prob1}
asks whether the classes of groups one obtains by considering $\Phi_G$
and $\widetilde{\Phi}_{G,\rho_\alpha}$ (for some/any fixed $\alpha\in (0,2)$) are 
all exactly the same.

Question 4 is technically interesting because we do not have 
good techniques to understand the subtle difference of behavior between 
$\widetilde{\Phi}_{G,\rho_2}$ and $\Phi_{G}$. We 
obtain a sharp answer for group of polynomial volume growth (see Corollary \ref{cor-Pol-Phi}) and 
for some wreath product (see Theorem \ref{th-wr-Pol-rho}).

\subsection{The spectral profiles  $\Lambda_G$ and 
$\widetilde{\Lambda}_{G,\rho}$}
Given a symmetric probability measure $\phi$, consider the associated 
Dirichlet form
$$\mathcal E_\phi(f_1,f_2)=\frac{1}{2}
\sum_{x,y\in G}(f_1(xy)-f_1(x))(f_2(xy)-f_2(x))\phi(y)$$
and set
$$\Lambda_{G,\phi}(v)=\Lambda_\phi(v)=\inf\{ 
\lambda_\phi(\Omega): \Omega\subset G,\; |\Omega|\le v\}$$
where
\begin{equation}\label{def-eig}
\lambda_\phi(\Omega)=
\inf\{ \mathcal E_\phi(f,f): \mbox{support}(f)\subset \Omega, \|f\|_2=1\}.
\end{equation}
In words, $\lambda_\phi(\Omega)$ is the lowest eigenvalue of 
the operator of convolution by $\delta_e-\phi$ with Dirichlet boundary
condition in $\Omega$. This operator is associated with the discrete 
time Markov process corresponding to the $\phi$-random walk killed 
outside $\Omega$.  The function $v\mapsto \Lambda_\phi(v)$ is called the 
$L^2$-isoperimetric profile or spectral profile of $\phi$ 
(it really is an iso-volumic profile). The $L^2$-isoperimetric profile 
of a group $G$ is defined as the 
$\simeq$-equivalence class $\Lambda_G$ of the functions
$\Lambda_\phi$ associated to any symmetric probability measure 
$\phi$ with finite generating support.  In Section \ref{sec-PhiLambda}, 
we give a short review of the 
well-known relations that exist between the behavior of 
$n\mapsto\phi^{(2n)}(e)$ and $v\mapsto \Lambda_\phi(v)$. 
It will be useful to introduce the following definition analogous to 
Definition \ref{def2}.
\begin{defin} Let $\rho, \widetilde{\mathcal S}_{G,\rho}$  be as in 
{\em Definitions
\ref{def1}-\ref{def2}}. Set
$$\widetilde{\boldsymbol \Lambda}_{G,\rho}(v)= 
\sup\left\{ \Lambda_{\phi}(v): \phi\in \widetilde{\mathcal S}_{G,\rho}\right\}, 
\;\; v>0.$$
\end{defin}
In words, $\widetilde{\boldsymbol \Lambda}_{G,\rho}$ is the 
extremal spectral profile under the weak $\rho$-moment condition. 
Upper bounds on $\widetilde{\boldsymbol \Lambda}_{G,\rho}$ are tightly related to 
lower bounds on $\widetilde{\Phi}_{G,\rho}$ and vice-versa. See Section
\ref{sec-PhiLambda}.

The appendix provides examples of the computation $\Lambda_\phi$ 
(and assorted $L^p$-variants) for radial stable-like 
probability measures defined in terms of the word distance.  

\subsection{Main results}
The goal of this work is twofold. 
First, we develop a new approach 
to obtain lower bounds on $\Phi_{G,\rho}$ and $\widetilde{\Phi}_{G,\rho}$. This 
method is simpler than the technique developed in \cite{BSClmrw} and 
is more generally applicable. In particular, the technique in \cite{BSClmrw} 
fails badly when the function $\rho$ grows too slowly (e.g., logarithmically).
In contrast,  the approach developed below provides good 
lower bounds on $\Phi_{G,\rho}$
for any increasing 
slowly varying function $\rho$ on any group $G$ for which one has a 
lower bound on $\Phi_G$. Second, we develop a method that allow us to obtain 
sharp upper bounds on 
$\widetilde{\Phi}_{G,\rho}$ in the context of wreath products. Here, 
we make essential use of earlier work of A. Erschler \cite{Erschler2006}. 
Our contribution is to develop a technique that allows us to harvest the 
$L^1$-isoperimetric results of Erschler in order to bound the random walk 
invariants $\widetilde{\Phi}_{G,\rho}$.  Both goals are attained by focusing
on the notion of isoperimetric profile (the $L^2$-isoperimetric profile but 
also the $L^p$ versions, $p\ge 1$, especially $p=1$).    

Our main results regarding the spectral profile $\Lambda_\phi$ and the
extremal profile $\widetilde{\boldsymbol \Lambda}_{G,\rho}$ are stated in 
Theorems \ref{pro-compphiG}-\ref{th-Lpsi}. Theorem
\ref{pro-compphiG} gives a general and easily applicable 
upper bound on $\Lambda_\phi$ in terms of $\Lambda_G$ under 
weak-moment conditions on $\phi$. Theorem \ref{th-Lpsi}
gives a completely satisfactory positive answer to a 
spectral profile version of Problem \ref{prob1}(1) 
for a large class of slowly varying functions $\rho$ including all
moment conditions of logarithm or iterated logarithm type. The following 
statement captures the nature of these results.
\begin{theo} Let $G$ be a finitely generated group equipped.
Let $\rho:[0,\infty)\ra [1,\infty)$ be a continuous increasing function.
The spectral profile functions $\Lambda_G$ and 
$\widetilde{\boldsymbol \Lambda}_{G,\rho}$ satisfy
$$\widetilde{\boldsymbol \Lambda}_{G,\rho} 
\lesssim \Lambda_{G} 
\int_0^{1/\Lambda^{1/2}_G}\frac{sds}{\rho(s)}.$$  
Further, if $\rho(t)\simeq 
\left(\int_{t}^\infty\frac{ds}{(1+s)\ell(s)}\right)^{-1}$
where $\ell$ is a slowly varying function satisfying 
$\ell(t^a)\simeq \ell(t)$ for all $a>0$, we have
$$ \widetilde{\boldsymbol \Lambda}_{G,\rho} \simeq \frac{1}{\rho(1/\Lambda_G)}.$$
In particular, for any $k\ge 1$ and $\epsilon>0$,
$\widetilde{\boldsymbol \Lambda}_{G,\log_{[k]}^\epsilon} (v)\simeq 1/
\log^\epsilon_{[k]}(1/\Lambda_G)$.
\end{theo}
The second part of this theorem indicates that, 
for slowly varying functions $\rho$ of the type described above, 
$\widetilde{\boldsymbol \Lambda}_{G,\rho}$ is 
determined by $\Lambda_G$.  Given the tight connections between 
$n\mapsto\phi^{(2n)}(e)$ and $v\mapsto \Lambda_\phi(v)$, this means that
 $\Phi_G$ determines $\widetilde{\Phi}_{G,\rho}$ under some 
a priori regularity conditions on these functions. 

In Section \ref{sec-entropy}, we derive a sublinear  upper bound for 
the entropy
$H_\mu(n)=\sum_{g\in G} (-\log \mu^{(n)}(g))  \mu^{(n)}(g)$ 
when the symmetric probability measure $\mu$ has finite $p$-moment 
and under an appropriate condition
on its $L^p$-isoperimetric profile which 
implies the following interesting result regarding the entropy and the
the displacement  $L_\mu(n)=\mu^{(n)}(|\cdot|)$. Compare to 
\cite[Theorem 1.4 and Conjecture 1.5]{Gournay}.
\begin{theo}
Let $G$ be a finitely generated infinite group such that
$$
\Phi_G(n)\gtrsim \exp\left(-n^{\gamma}\right)
$$
for some $\gamma\in\left(0,\frac{1}{2}\right)$.
Let $\mu$ be a symmetric probability measure on $G$ with finite
$p$-moment where $p> 2\gamma/(1-\gamma)$ and $p>1$.
Then, for any fixed $\epsilon>0$, 
\[
H_{\mu}(n)\lesssim
\left(n(\log n)^{1+\epsilon}\right)^{\frac{2\gamma}{p(1-\gamma)}}.
\]
In particular, if $p=2$, we have $
H_{\mu}(n)\lesssim
\left(n(\log n)^{1+\epsilon}\right)^{\frac{\gamma}{(1-\gamma)}}$ 
and 
\[
L_{\mu}(n)\lesssim n^{\frac{1}{2(1-\gamma)}}(\log n)^{\frac{(1+\epsilon)\gamma}{2(1-\gamma)}}.
\]
\end{theo}
The last conclusion follows from the entropy bound by 
\cite[Corollary 1.1]{EK} which gives the bound 
$L_\mu(n)\lesssim \sqrt{nH_\mu(n)}$ assuming that $\mu$ is symmetric and 
has finite second moment.

Section \ref{sec-lowPhi} describes applications of the spectral profile
upper bound provided by Theorem \ref{pro-compphiG} to the problem of 
bounding $\widetilde{\Phi}_{G,\rho}$ from below. The main result is Theorem 
\ref{th-rho-low} which gives sharp lower bounds on $\widetilde{\Phi}_{G,\rho}$ 
in terms of a lower bound on $\Phi_G$ for a wide variety of weak-moment 
conditions. One important feature of this result 
(which distinguished it from the results obtained in \cite{BSClmrw}) is that 
it is just as effective around the critical weak-moment condition 
of order $2$ than for power weak-moment conditions in the classical
range $(0,2)$ (stable like moment-conditions) and for moment conditions 
associated with slowly varying functions (including positive powers of any 
iterated logarithms).  Explicit statements are given in Corollaries 
\ref{cor-Pol-Phi}-\ref{cor-Good-Phi}-\ref{cor-low-Phi1}.

Section \ref{sec-wreath} is devoted to wreath products. These groups are 
important for many reasons including the fact that they provide a class 
of groups of exponential volume growth in which a rich variety of different 
behaviors  of $\Phi_G$ occurs.  Here, we provide
sharp upper bounds on $\widetilde{\Phi}_{G,\rho}$. 
More generally, we provide sharp two-sided bounds on $n\mapsto \phi^{(2n)}(e)$
for a wide variety of measures $\phi$ on wreath products and iterated 
wreath products. For instance, let $G=\mathbb Z_2\wr H$ be the  
lamplighter group with the usual binary lamps over a based group $H$ which 
has polynomial volume growth of degree $D$. In this simple case, we obtain the 
following estimates 
$$\widetilde{\Phi}_{G,\rho_2}(n)\simeq \exp\left(- (n\log n )^{d/(d+2)}\right),$$
$$\widetilde{\Phi}_{G,\rho_\alpha}(n)\simeq 
\exp\left(- n^{d/(d+\alpha)}\right),\;\;\alpha\in (0,2)$$
$$\widetilde{\Phi}_{G,\log_{[k]}^\epsilon}(n)\simeq 
\exp\left(- n/\log^\epsilon_{[k]}(n)\right),\;\;k=1,2\dots, \epsilon>0.$$
The first and last estimates appear to be new even for $H=\mathbb Z$. When 
$H=\mathbb Z^d$, the second estimate can be derived from the celebrated 
Donsker-Varadhan large deviation theorem on the number of visited point 
by a random walk that belongs to the domain of attraction of a stable law.
These results follow from the techniques developed in Section 
\ref{sec-wreath} and are part of a large collection of illustrative examples 
described in Section \ref{sec-wreathexa}.

A key result concerning wreath products is 
Erschler's isoperimetric inequality \cite{Erschleriso,Erschler2006} which gives 
a lower bound on the F\o lner function of $G=H_1\wr H_2$ in terms of the 
F\o lner functions of of $H_1$ and $H_2$. Here we use Erschler's inequality 
and a new comparison idea to obtain the following  $L^2$-version where
$\Lambda^{-1}_{H,\mu}$ denotes the right-continuous  inverse of the 
$L^2$-isoperimetric profile $\Lambda_{H,\mu}$.  
\begin{theo}\label{th-WR}
There exists a constant $K>0$ such that, for any 
symmetric probability measures $\mu_{H_1}$ and $\mu_{H_2}$  
on two finitely generated groups $H_1$ and $H_2$, 
the switch or walk measure  
$q=\frac{1}{2}(\mu_{H_1}+\mu_{H_2})$ on the wreath product $H_1\wr H_2$ satisfies 
\[
\Lambda_{H_1\wr H_2,q}(v)\ge s/K\ \mbox{for any }v\le\Lambda_{H_1,\mu_{H_1}}^{-1}(s)^{\Lambda_{H_2,\mu_{2}}^{-1}(s)/K}.
\]
In the other direction, 
\[
\Lambda_{H_1\wr H_2,q}(v)\le s\ \mbox{for }v\ge\Lambda_{H_1,\mu_{H_1}}^{-1}(s)^{\Lambda_{H_2,\mu_{2}}^{-1}(s)}\Lambda_{H_2,\mu_{2}}^{-1}(s).
\]
\end{theo}
The only case where this result is far from sharp is when either 
$H_1=\{e\}$ is trivial or $H_2$ is finite. In those cases, 
it is a simple matter to obtain the desired sharp results by different 
arguments.  Because of the detailed relations between 
the $L^2$-isoperimetric profile $\Lambda_\phi$ and the behavior of 
$n\mapsto \phi^{(2n)}(e)$ (see Section \ref{sec-PhiLambda}), 
the above theorem typically yields good bounds on $q^{(2n)}(e)$
in terms of bounds on $n\mapsto \mu_1^{(2n)}(e)$ and $n\mapsto \mu_2^{(2n)}(e)$.

\section{The isoperimetric functions $\Lambda_{p,\phi}, p\ge 1$}
\setcounter{equation}{0}
\subsection{$\Phi$, $\Lambda$ and the Nash profile}\label{sec-PhiLambda}

In this section, we quickly review Coulhon's results from \cite{CNash}
which, in the present context, relate the behavior of 
$n\mapsto \phi^{(2n)}(e)$ to that of the spectral profile
$v\mapsto \Lambda_\phi(v)$. We refer the reader to \cite{CNash} for 
references to earlier related works, in particular, work by Grigor'yan 
in which the spectral profile play a key role.

It is convenient to introduce the notion of Nash profile.
Namely, define the Nash profile  $\mathcal N_A$ of a 
symmetric Markov generator $A$ with associated Dirichlet form $\mathcal E_A$ 
by
$$\mathcal N_A(t)=\sup\left\{\frac{\mathcal E_A(f,f)}{\|f\|_2^2}: 
f\in \mbox{Dom}(\mathcal E_A) 
\mbox{ with } 0<\|f\|^2_1\le t\|f\|^2_2\right\}$$
so that, for all $f$ in the domain of the Dirichlet form $\mathcal E_A$,
$$\|f\|_2^2\le \mathcal N_A(\|f\|_1^2/\|f\|_2^2) \mathcal E_A(f,f).$$
For our purpose, we can restrict ourselves 
to the case when $A$ is convolution by 
$\delta_e-\phi$, for some symmetric probability measure $\phi$ on $G$. 
In this case, with some abuse of notation, $\mathcal E_A=\mathcal E_\phi$,
$\mbox{Dom}(\mathcal E_A)=L^2(G)$ and we will write 
${\mathcal N}_\phi$ for the Nash profile of
$A =\cdot*(\delta_e-\phi)$. 
The following lemma relates the $L^2$-isoperimetric profile and 
the Nash Profile. 
\begin{lem}[Folklore] \label{lem-IsoNash}
For any symmetric probability measure $\phi$ on a (countable) group $G$, 
we have
$$\forall\, v>0,\;\;  \frac{1}{\Lambda_\phi( v)} \le  
{\mathcal N}_\phi(v) \le \frac{2}{\Lambda_\phi (4v)}.$$
\end{lem}
\begin{proof} For any finite set $\Omega$ and any function $f$ 
with support in $\Omega$, $\|f\|_1^2\le |\Omega|\|f\|_2^2$. 
Hence,the lower bound on $\mathcal N_\phi$ follows easily from the definitions
of $\mathcal N_\phi$ and $\Lambda_\phi$. 
Conversely, the definition of $\Lambda_\phi$ gives
$$\|f\|_2^2\le \Lambda_\phi(|\mbox{support}(f)|)^{-1}\mathcal E_\phi(f,f).$$
For any $t\ge 0$, set $f_t=\max\{f-t,0\}$ and observe that, for any 
non-negative $f$, $|f|^2\le (f_t)^2+ 2t f$ and 
$\mathcal E(f_t,f_t)\le \mathcal E(f,f)$. 
It follows that, for any $t$ and $f\ge 0$,
$$\|f\|_2^2\le \Lambda_\phi(|\{f\ge t\}|)^{-1} \mathcal E(f,f)+2t\|f\|_1.$$
Picking $t$ such that $4t=\|f\|_2^2/\|f\|_1$ and using 
$|\{f\ge t\}|\le t^{-1}\|f\|_1$, we obtain
$$\|f\|_2^2\le 2\Lambda_\phi( 4\|f\|_1^2/\|f\|_2^2)^{-1}\mathcal E(f,f).$$
The upper bound on $\mathcal N_\phi$ follows. 
\end{proof}

\begin{theo}[Essentially, {\cite[Proposition II.1]{CNash}}] 
\label{th-Coulhon1}
We have
$$\phi^{(2n+2)}(e) \le  2\psi(2n)$$
where 
$\psi: [0,+\infty)\ra [1,+\infty)$ is defined implicitly by
$$t=\int_1^{1/\psi(t)}\frac{ds}{2s\Lambda_\phi(4s)}.$$
\end{theo}

\begin{proof} It is convenient to observe that
$$\phi^{(2n+2)}(e)\le 2 h^\phi_{2n}(e) \mbox{ and } 
h^\phi_{4n}(e)\le e^{-2n}+\phi^{(2n)}(e)$$
where 
\begin{equation}\label{cont-time}
h^\phi_t=
e^{-t}\sum_0^\infty\frac{t^k}{k!}\phi^{(k)}.\end{equation} 
See, e.g., \cite[Section 3.2]{PSCstab}. Convolution by $h^\phi_t$ 
defines the continuous time semigroup associated with 
the continuous time random walk driven by $\phi$.  
Lemma \ref{lem-IsoNash} gives us the Nash inequality
$$\|f\|_2^2\le 2\Lambda_\phi( 4\|f\|_1^2/\|f\|_2^2)^{-1}\mathcal E(f,f).$$
Using this inequality in the Proof 
of \cite[Proposition II.1]{CNash} gives
$h^\phi_t(e)\le  \psi(t).$
\end{proof}

The following is a sort of converse of Theorem \ref{th-Coulhon1}.
\begin{theo}[{\cite[Proposition II.2]{CNash}}] \label{th-Coulhon2}
For $v\ge 1$,
$$\Lambda_\phi(v) \ge 
\sup_{t>0}\left\{ \frac{1}{2t}\log\frac{1}{vh^\phi_t(e)}\right\}.$$
\end{theo}

\begin{rem}Assume that
$\psi$ is a continuous decreasing function with continuous derivative 
with the property that
there exists $\epsilon>$ such that for all $t>0$ and all $s\in (t,2t)$
we have  $$\frac{-\psi'(s)}{\psi(s)}\ge \epsilon 
\frac{\psi'(t)}{\psi(t)}.$$
As noted in \cite{CNash} and elsewhere, under this condition
the functions $$x\mapsto \Lambda(x)=
\sup_{t>0}\left\{ \frac{1}{2t}\log\frac{1}{x\psi(t)}\right\} \mbox{ and } 
x\mapsto - \frac{1}{x} \psi'\circ \psi^{-1}(x)$$ are $\simeq$-equivalent.
Hence, under this regularity condition on $\psi$, 
$ n\mapsto \phi^{(2n)}(e)\lesssim \psi $ is equivalent to 
$\Lambda\lesssim \Lambda_\phi.$
\end{rem}

The following lemma will be useful later.

\begin{lem} \label{lem-phiLambda}
Assume that $\phi^{(n)}(e)\ge \exp(-n/\pi(n))$ where 
$\pi:(0,\infty)\ra (0,\infty)$ is an increasing function satisfying 
$\pi(t)\le ct$.  Then there exists $A$ such that for all $n$ we have
$$\Lambda_\phi(\exp(An/\pi(n))\le \frac{A}{2\pi(n)}.$$
\end{lem}
\begin{proof} 
Let $\psi$ be defined in terms of $\Lambda_\phi$ as in Theorem 
\ref{th-Coulhon1}. By definition and since $\Lambda_\phi$ is a 
non-increasing function, we have
$$t\le \frac{\log(1/\psi(t))}{2\Lambda(1/\psi(t))}$$
which we rewrite as
$$\Lambda_\phi(1/\psi(t))\le \frac{\log(1/\psi(t))}{2t}.$$
By Theorem \ref{th-Coulhon1} and the hypothesis, for $A$ large enough,  
$$\exp(-An/\pi(n))\le \psi(n).$$ Hence
$$\Lambda_\phi(\exp(An/\pi(n)))\le \Lambda_\phi(1/\psi(n))\le 
\frac{A}{2\pi(n)}.$$
\end{proof}
\begin{rem}In most cases, $n\ra An/\pi(n)$ is invertible and the Lemma 
gives an upper-bound on $\Lambda_\phi$.
\end{rem}

\begin{cor}[Folklore] Let $\phi$ be a symmetric probability measure on $G$.
\begin{itemize}
\item If
$\phi^{(n)}(e)\ge \exp(-n^\gamma)$. Then $\Lambda_\phi(v)\le C[\log(e+ v)]
^{-(1-\gamma)/\gamma}$.
\item If $\phi^{(n)}(e)\ge \exp(-n/[\log n]^\gamma)$. 
Then $\Lambda_\phi(v)\le C[\log(e+\log(e+ v))]^{-\gamma}$.
\end{itemize}
\end{cor}

\subsection{The profiles $\Lambda_{p,\phi}$}
The $L^2$-isoperimetric profile $\Lambda_{2,\phi}=\Lambda_\phi$
is naturally related to the analogous $L^1$-profile 
$$\Lambda_{1,\phi}(v)=\inf\left\{\frac{1}{2}\sum_{x,y}|f(xy)-f(x)|\phi(y): 
|\mbox{support}(f)|\le v,\;\;\|f\|_1=1\right\}.$$
Using an appropriate discrete co-area formula, 
$\Lambda_{1,\phi}$ can equivalently be defined by
$$\Lambda_{1,\phi}(v)= \inf
\left\{|\Omega|^{-1}\sum_{x,y}\mathbf 1_\Omega(x)
\mathbf 1_{G\setminus \Omega}(xy)\phi(y): 
|\Omega|\le v\right\}.$$
If we define the boundary of $\Omega$ to be the set   
$$\partial\Omega=\left\{(x,y)\in G\times G: x\in \Omega, y\in G\setminus \Omega
\right\}$$ and set
$\phi(\partial \Omega)= \sum_{x\in \Omega,xy\in G\setminus \Omega}\phi(y)$
then $\Lambda_{1,\phi}(v)= \inf\{\phi(\partial\Omega)/|\Omega|: |\Omega|\le v\}$.

From these definitions and remarks, it follows that
\begin{equation} \label{Cheeger}
\frac{1}{2}\Lambda^2_{1,\phi}\le \Lambda_{2,\phi}\le \Lambda_{1,\phi}
\end{equation}
The upper bound is very straightforward since it suffices to test the 
definition of $\Lambda_{2,\phi}$ on functions of the type $\mathbf 1_\Omega$
to obtain it. The lower bound is obtained by testing the definition 
of $\Lambda_{1,\phi}$ on functions of the form $f^2$, $f\ge 0$, and using the 
Cauchy-Schwarz inequality. In fact, for any $p\ge 1$, set
$$\mathcal E_{p,\phi}(f)=\frac{1}{2}\sum_{x,y}|f(xy)-f(x)|^p\phi(y)$$
and 
\begin{equation}\label{p-profile}
\Lambda_{p,\phi}(v)=
\inf
\left\{ \mathcal E_{p,\phi}(f): 
|\mbox{support}(f)|\le v, \|f\|_p=1\right\}.
\end{equation}
\begin{pro}[Folklore] For $1\le p\le q<\infty$, we  have
\begin{equation} \label{Cheegerpq}
c(p,q)\Lambda^{q/p}_{p,\phi}\le \Lambda_{q,\phi} \le C(p,q)\Lambda_{p,\phi}. 
\end{equation}
\end{pro}
\begin{proof} This is closely related but different from 
\cite[Corollaire 3.2]{CSCiso}. The inequality
$c(p,q)\Lambda^{q/p}_{p,\phi}\le \Lambda_{q,\phi} $, $1\le p\le q$, 
which is a form of Cheeger's inequality, 
is obtained by testing  $\Lambda_{p,\phi}$ on functions of the 
form $f^{q/p}$, $f\ge 0$, and using H\"older inequality.  The inequality
$\Lambda_{q,\phi} \le C(p,q)\Lambda_{p,\phi}$ can be proved as follows.

For any function $f\ge 0$, set 
$f_k=(f-2^k)^+\wedge 2^k$, $k\in \mathbb Z$. 
By \cite[Section 6]{BCLS}, we have
\begin{equation}\label{Hbcls}
\left(\sum_k \mathcal E_{p,\phi}(f_k) ^{\alpha/p}\right)^{1/\alpha}\le 
2(1+p) \mathcal E_{p,\phi}(f). \end{equation}
This should be understood as an $L^p$ substitute for the $L^1$
co-area formula. 

Now, if we assume that $|\mbox{support}(f)|\le v$, we have
$$  \Lambda_{q,\phi}(v) \|f_k\|_q^q\le \mathcal E_{q,\phi}(f_k) .$$
Noting that $f_k\ge 2^k$ on $\{f\ge 2^{k+1}\}$ and that $0\le f_k/2^k\le 1$,
we obtain
$$ \Lambda_{q,\phi}(v)
2^{k q} |\{ f\ge 2^{k+1}\}|\le  2^{k(q-p)} \mathcal E_{p,\phi}(f_k).$$
This gives
\begin{equation}\label{weakLambda}
 \Lambda_{q,\phi}(v)
2^{(k+1) p} |\{ f\ge 2^{k+1}\}|
\le  2^p \mathcal E_{p,\phi}(f_k).\end{equation}
It is easy to check that (see, e.g., \cite[(4.2)]{BCLS}) 
$$\|f\|_p^p\le 2^p \sum _k 2^{(k+1) p} |\{ f\ge 2^{k+1}\}|.$$ 
Using (\ref{weakLambda}) and  (\ref{Hbcls}), this yields
$$\Lambda_{q,\phi}(v)\|f\|_p^p\le  2(1+p) 4^p \mathcal E_{p,\phi}(f)$$
Optimizing over all $f$ implies  that
$\Lambda_{q,\phi}(v)\le  2(1+p) 4^p \Lambda_{p,\phi}(v)$ as desired.
\end{proof}

\subsection{Entropy upper bounds using $\Lambda_{p,G}$ upper bounds}
\label{sec-entropy}

Given a probability measure $\mu$ on $G$, its entropy function $H_\mu$
is defined by
$$H_\mu(n)=\sum_{g\in G} -(\log \mu^{(n)}(g))  \mu^{(n)}(g).$$
See, e.g., \cite{Erschlerdrift0,Erschlerdrift,KV}.  Recall that $\mathbf u$ 
denotes the uniform probability measure on the 
symmetric generating set $S^*$ (by definition, $S^*$ contains the
identity element). Also, consider the displacement function 
$$n\mapsto L_\mu(n)= \sum_{g\in G}|g|\mu^{(n)}(g).$$

\begin{theo} \label{theo-entropy}
Assume that there exist $p\in (1,2]$, 
$\alpha\in (0,1)$ and an increasing slowly varying function $\ell$ 
such that
$$\Lambda_{p,\mathbf u}(v)\le \frac{\ell(\log(e+v))}{\log (e+v)^{1/\alpha}} .$$
For any symmetric probability measure $\mu$ with a finite $p$-moment
$\sum |g|^p\mu(g)<\infty$, we have
\begin{equation}\label{entropybound}
H_\mu(n)\le C(\mu,p,\omega)\; n^\alpha \omega(n)\end{equation}
for any increasing slowly varying function  $\omega$ such that 
$$ \ell( y^{\alpha\eta})^\alpha [\log (e+y^{\alpha\eta})]^{\alpha(1+\epsilon)}
\le \omega(y)$$
for some $\eta>1$ and $\epsilon>0 $.  Further
\begin{equation}\label{dispbound}
L_\mu(n)\le C(\mu,p,\omega)\; n^{[(3-p)+\alpha(p-1)]/2} \omega(n)^{p-1)/2}.
\end{equation}
\end{theo}
\begin{rem} Assume that the group $G$ satisfies
$\Phi_G(n)\ge \exp(-n^\gamma)$ where $\gamma\in (0,1)$. 
By Lemma \ref{lem-phiLambda}, we have 
$\Lambda_{2,G}(v) \le C[\log (e+v)]^{-(1-\gamma)/\gamma}$ and 
$$\Lambda_{p,G}(v)\le C[\log (e+v)]^{-\alpha_p},\;\;\alpha_p= \frac{2\gamma}{
p(1-\gamma)},\;\;p\in [1,2].$$ 
If $\gamma\in (0,1/2)$ then $2\gamma/(1-\gamma)<2$. For any $p>1$  
such that $2\gamma/(1-\gamma)<p\le 2$, we have
$\alpha_p= 2\gamma/(p(1-\gamma))
 \in (0,1) $. Under these hypotheses, for any symmetric
measure $\mu$ with finite $p$-moment,
Theorem 
\ref{theo-entropy} gives 
$$H_\mu(n)\le C_\mu n^{\alpha_p} (\log n)^{\alpha_p(1+\epsilon)}$$
for any $\epsilon>0$. In particular, the entropy of $\mu$ is sublinear and 
the entropy criteria \cite[Theorem 1.1]{KV} implies that 
bounded $\mu$-harmonic functions must be constant. 
\end{rem}
\begin{rem} The lamplighter group $G=\mathbb Z_2\wr \mathbb Z^2$
satisfies $\Phi_{G}(n)\simeq \exp(-n^{1/2})$ (equivalently,
$\Lambda_G(v)\simeq \log (e+v)^{-1}$) and $H_{\mathbf u}(n)\simeq n/\log n$.
See \cite{Erschlerdrift,Erschler2006,PSCwp}. This example is just beyond the 
limit of application of our result. Kotowski and Vir\'ag \cite{KV-1/2} describes a group $G$ for which $\Phi_G(n)\gtrsim \exp(-n^{1/2 +o(1)})$ 
and for which simple random walk has linearly growing entropy 
(the group has non-trivial bounded harmonic functions).  
\end{rem} 
\begin{rem} Theorem \ref{theo-entropy} is related to 
\cite[Theorem 1.4]{Gournay} 
and  \cite[Conjecture 1.5]{Gournay}
and some of the ingredients of the proof 
given below are similar to those used in \cite{Gournay}. 
The hypothesis (OD) appearing in \cite[Theorem 1.4]{Gournay} 
plays no role in Theorem \ref{theo-entropy}. 
\end{rem}

\begin{proof}
The proof of (\ref{entropybound}) uses the embedding of $G$ into a 
$L^p$ space introduced in 
\cite{TesseraBanach} together with \cite[Theorem 2.1]{NaorPeresIMRN}.

For each $k$, let $\phi_k$ be a function supported in a set $U_k$ 
of size $2^{2^k}$ and such that
$$\Lambda_{p,\mathbf u}(2^{2^k})=\inf\left\{\frac{1}{2}
\sum_{x,y}|f(x)-f(xy)|^p\mathbf u(y):
|\mbox{support}(f)|\le 2^{2^k}, \|f\|_p=1\right\}$$ 
is greater than 
$$\frac{1}{4} \frac{\sum_{x,y}|\phi_k(x)-\phi_k(xy)|^p\mathbf u(y)}{\sum_x|\phi_k(x)|^p}.$$
Let $\mathbf B_p(G)$ be the Banach space of sequences $(w_k)_1^\infty$ 
of elements of $\ell^p(G)$ such that $\sum_k \|w_k\|_p^p<\infty$ equipped with the norm $\|w\|_p=\left(\sum_k \|w_k\|_p^p\right)^{1/p}.$
 
Consider the embedding $b$ of the group $G$ into 
$\mathbf B_p(G)$ defined by
$$b(g)= \left(c_k \frac{\phi_k-\tau_r(g)\phi_k}
{\mathcal E_p(\phi_k)^{1/p}}\right)_1^\infty, \;\;
c_k^p=\frac{1}{(1+k)^{1+\epsilon}},$$
where 
$\tau_r(g)f: x\mapsto f(xg)$ is right translation by $g$ and
$$\mathcal E_p(f)=\frac{1}{2}\sum_{x,y}|f(xy)-f(x)|^p\mathbf u(y).$$
By construction, this is a $1$-cocycle, more precisely, an element of
$Z^1(G,\tau_r,\mathbf B_p(G))$. Indeed, for each $g\in G$,
$b(g)$ belongs to $\mathbf B_p(G)$ 
because 
\begin{equation}\label{ppp}
\|f-\tau_r(g)f\|^p_p\le |S||g|^p \sum_{x,y}|f(xy)-f(x)|^p\mathbf u(y)
\end{equation}
and  $\sum c_k^p<\infty$.
Set $\Omega_k= [\cup_1^k U_i]^{-1} [\cup_1^kU_i]$ and $\Omega_0=\emptyset$.
Note that for $g\not\in \Omega_k$, the functions $\phi_k$ and 
$\tau_r(g)\phi_k$ have disjoint supports and write
\begin{eqnarray}
\|b(g)\|_p^p& \ge & \sum_1^\infty c_k^p \frac{2\|\phi_k\|_p^p}{\mathcal E_p(\phi_k)}
\mathbf 1_{\{G\setminus \Omega_k\}}(g) \nonumber \\
&\ge&  \sum_1^\infty c_k^p \frac{1}{\Lambda_{p,\mathbf u}(2^{2^{k}})}
\mathbf 1_{\{\Omega_{k+1}\setminus \Omega_k\}}(g). \label{lowb}
\end{eqnarray}

By a well-known convexity argument,
$$H_\mu(n)\le \mathbf E_\mu\left(\sum _1^\infty \log |\Omega_k\setminus \Omega_{k-1}|\;\mathbf 1_{\Omega_k\setminus \Omega_{k-1}}(S_n)\right).$$  
By hypothesis,
$$ \log (e+|\Omega_{k+1}|) \simeq 2^k \simeq \log (e+ |\Omega_k|) \le 
C\frac{\ell( 2^k)}{\Lambda_{p,\mathbf u}(2^{2^k})^{\alpha}}.$$  
Let $F:[0,\infty)\ra [0,\infty)$ be a concave increasing 
function with $F(0)=0$
and such that
$$a t^\alpha \omega(t)
\le F(t) \le A t^\alpha \omega(t)$$
where $\omega$ is as in the statement of the theorem.
As $c_k^p=(1+k)^{-1-\epsilon}$, one can check that
$$ \frac{\ell( 2^k)}{\Lambda_{p,\mathbf u}(2^{2^k})^{\alpha}} 
 \le C F\left(\frac{c_k^p}{\Lambda_{p,\mathbf u}(2^{2^k})}\right)$$
 from which it follows that (with a different constant $C$)
$$ \log |\Omega_{k+1}| 
\le C F\left(\frac{c_k^p}{\Lambda_{p,\mathbf u}(2^{2^k})}\right).$$
Now, we have
\begin{eqnarray*}
\lefteqn{
\mathbf E_\mu\left(\sum _1^\infty \log |\Omega_{k+1}|\;
\mathbf 1_{\Omega_{k+1}\setminus \Omega_{k}}(S_n)\right)}
\hspace{1in} &&\\
&\le & 
C\mathbf E_\mu\left(\sum _1^\infty 
F\left(c_k^p /\Lambda_{p,G}(2^{2^k})\right)
\mathbf 1_{\Omega_{k+1}\setminus \Omega_{k}}(S_n)\right)\\
&= &C\mathbf E_\mu\left(
F\left(\sum _1^\infty c_k^p /\Lambda_{p,G}(2^{2^k})
\mathbf 1_{\Omega_{k+1}\setminus \Omega_{k}}(S_n)\right)\right)\\
&\le & C F \left(\mathbf E_\mu \left(\sum _1^\infty 
\frac{c_k^p }{\Lambda_{p,G}(2^{2^k})} 
\mathbf 1_{\Omega_{k+1}\setminus \Omega_{k}}(S_n)\right)\right)\\
&\le &
C F \left(\mathbf E_\mu \left(\|b(S_n)\|_p^p\right)\right).
\end{eqnarray*}
where the second to last inequality is Jensen's inequality applied to 
the concave function $F$. Finally,
we appeal to \cite[Theorem 2.1]{NaorPeresIMRN} and (\ref{ppp}) 
to conclude that (since $1<p\le 2$)
$$\mathbf E_\mu(\|b(S_n)\|_p^p)\le C_p n \mathbf E_\mu(|b(S_1)|^p)
\le C(p,S) n \sum_x|x|^p\mu(x).$$
The statement in \cite{NaorPeresIMRN} is for simple random 
walk but the proof works for an arbitrary symmetric measure $\mu$ with 
finite $p$-moment. Note that $p>1$ is essential here. This finishes 
the proof of the entropy bound (\ref{entropybound}).

We now explain how (\ref{dispbound}) follows from (\ref{entropybound}).
The statement in \cite[Corollary 5.2(i)]{EK} gives the bound 
$$L_\mu(n)\le C\sqrt{nH_\mu(n)}$$
under the assumption that the symmetric probability  
measure $\mu$ has finite second moment. This follows from two bounds
\begin{enumerate}
\item[(a)]  $|L_\mu(n+1)-L_\mu(n)|\le C\beta(n)$ (\cite[Corollary 5.2(i)]{EK}), 
\item[(b)]  $\beta(n)\le C\sqrt{H_\mu(n+1)-H_\mu(n)}$ (\cite[Lemma 5.1(ii)]{EK})
\end{enumerate}
 where 
$$\beta(n)= \sup_{s\in S}
\left\{\sum_{g\in G} |\mu^{(n)}(gs)-\mu^{(n)}(g)|\right\}.$$
The hypothesis that $\mu$ has second moment enters (a) 
but is not necessary for (b). If we replace the hypothesis that 
$\mu$ has finite second moment by the hypothesis that $\mu$ as 
finite weak-$p$-moment $W(\rho_p,\mu)<\infty$ for some $p\in (1,2]$, 
an easy modification of the proof of (a) given in \cite{EK} gives
\begin{enumerate}
\item[(a')]  $|L_\mu(n+1)-L_\mu(n)|\le C\beta(n)^{p-1}$.
\end{enumerate}
Set $\theta= (p-1)/2 \in (0,1/2]$ and write 
\begin{eqnarray*}
L_\mu(n)
&=& \sum_1^n 
\frac{L_\mu(j)-L_\mu(j-1)}{\beta(j)^{p-1}}
\beta(j)^{p-1}\\
&\le & 
\left(\sum_1^n 
\left(\frac{L_\mu(j)-\L_\mu(j-1))}{\beta(j)^{p-1}}\right)^{1/(1-\theta)}\right)^{1-\theta}
\left(\sum_1^n \beta(j)^2\right)^\theta\\
&\le & C n^{1-\theta} H_\mu(n)^\theta=C n^{(3-p)/2} H_\mu(n)^{(p-1)/2}.
\end{eqnarray*} 
This shows that (\ref{dispbound}) follows from (\ref{entropybound}).
\end{proof}

\subsection{Comparison of $\Lambda_{p,\phi}$ with $\Lambda_{p,\mathbf u}$}

By definition, we let $\Lambda_{p,G}$ 
be the $\simeq$-equivalence class of 
$\Lambda_{p,\phi}$
when $\phi$ is a fixed symmetric measure with finite 
generating  support on $G$. Note that $\Lambda_{p,G}$  
does not depend on the choice of $\phi$. We refer to this case as 
the classical case.

This subsection is devoted to a simple yet very useful result that provides 
upper bounds for $\Lambda_{p,\phi}$, $p\ge 1$ 
in terms of $\Lambda_{p,G}$ and basic information on the probability 
measure $\phi$. We can represent $\Lambda_{p,G}$ by 
$\Lambda_{p,\mathbf u}$ where $\mathbf u$ is 
the uniform measure on the fixed generating finite symmetric set $S^*$. 

For any increasing continuous function $\rho:[0,\infty)\ra [1,\infty)$, set
\begin{equation}
M_{p,\rho}(t)=t^p\left(\int_0^t \frac{s^{p-1}}{\rho(s)}ds\right)^{-1}. 
\end{equation} 
Note that we always have
$$M_{p,\rho}(t)\le p\, \rho(t).$$
Further, when $\rho$ is regularly varying of index $\alpha\in [0,\infty)$,
we have $M_{p,\rho} \simeq \rho$ if $\alpha\in [0,p)$ and 
$M_{p,\rho}(t)\simeq t^p$ if $\alpha>p$. In the case $\alpha=p$, 
explicit computations are necessary. For instance, when $\rho(t)=(1+t)^p$, 
$M_{p,\rho}(t)\simeq 1+\log (1+t).$

The following theorem will be used to obtain good lower bounds
on $\widetilde{\Phi}_{G,\rho}$, in particular, when $\rho$ is a slowly growing function.

\begin{theo} \label{pro-compphiG}
Let $\phi$ be a symmetric probability measure satisfying the 
weak moment condition
$$W(\rho,\phi)=\sup_{s>0}\{s\phi(\{x:\rho(|x|)>s\})\}\le K.$$
Then for any $v>0$ and $p\in [1,\infty)$, we have
$$\Lambda_{p,\phi}(v)\le C(p,\rho)K
\inf_{s>0}\left\{ \frac{1}{\rho(s)}+ \frac{|S^*|s^p}{M_{p,\rho}(s)}
\Lambda_{p,\mathbf u}(v)\right\}.
$$ 
In particular
$$\Lambda_{p,\phi}(v)\le 
\frac{C(p,\rho, |S^*|,K) }{M_{p,\rho}
(\Lambda_{p,\mathbf u}(v)^{-1/p})}.
$$ 
\end{theo}
\begin{proof}Recall that  
$$\Lambda_{p,\phi}(v)= \inf
\left\{\frac{1}{2}\sum_{x,y}|f(xy)-f(x)|^p\phi(y):|\mbox{support}(f)|\le v,
\|f\|_p=1\right\}$$

For any function $f$, write
\begin{eqnarray}
\sum_{x,y}|f(xy)-f(x)|^p\phi(y)
&=&\sum_{x}\sum_{|y|\le s}|f(xy)-f(x)|^p\phi(y) \nonumber\\
&& +
\sum_{x}\sum_{|y|>s}|f(xy)-f(x)|^p\phi(y) \label{split}
\end{eqnarray}
Making use of the well-known (pseudo-Poincar\'e) inequality (\cite{CSCiso})
$$\forall\,y\in G,\;\;
\sum_x|f(xy)-f(x)|^p\le |S^*||y|^p\sum_{x,z}|f(xz)-f(x)|^p\mathbf u(z)$$
the first right-hand term is bounded by
$$ |S^*|\left(\sum_{|y|\le s}|y|^p\phi(y)\right)\sum_{x,y}|f(xy)-f(x)|^p\mathbf u(y).$$ 
Further,
\begin{eqnarray*}
\sum_{|y|\le s}|y|^p\phi(y)&=&p\sum_{0\le k\le s} (k+1)^{p-1}\phi(\{x:|x|> k\})\\
&\le&  pK \sum_{0\le k\le s}\frac{(1+k)^p}{\rho(k)} \le C(p,\rho)K s^p/M_{p,\rho}(s).  
\end{eqnarray*}
To bound the second term on the right-hand side of (\ref{split}), we let
$$\phi'_s(y)=(\phi(\{x:|x|>s\}))^{-1}\phi(y)\mathbf 1_{\{|\cdot|>s\}}(y)$$ and
write
$$\sum_{x}\sum_{|y|>s}|f(xy)-f(x)|^p\phi(y)\le \phi(\{x:|x|>s\})
\mathcal E_{\phi'_s}(f,f)\le \frac{2K}{\rho(s)}\|f\|_2^2. $$
The desired inequality follows (with an adjusted constant $C(p,\rho)$).
\end{proof}
\subsection{Subordination}
This section introduces notation and results regarding the notion of 
subordination. We will use this notion in several important ways. 
For more background and further references to the 
literature, see \cite{Bendikov2012,BSClmrw}.

Recall that a Bernstein function is a function $f:(0,\infty)\ra \mathbb R$
such that 
\begin{equation}\label{Bern}
f(s)=a+bs +\int_{(0,\infty)}(1-e^{-st}) \nu(dt) \end{equation}
where $a,b\ge 0$ and  $\nu$ is a measure satisfying $\int_{(0,2)}d\nu(dt)
+\int_{(1,\infty)}\nu(dt)<\infty$. The measure $\nu$ is called the L\'evy 
measure of $f$.  See \cite{SSV} for details. 
For our purpose, it suffices to consider the case  $a=b=0$.  
The most classic example of Bernstein function is
$s\mapsto s^\alpha$, $\alpha\in (0,1)$, which 
has $\nu(dt)= \alpha \Gamma(\alpha-1)^{-1} t^{-1-\alpha} d t$.

A complete Bernstein function is a function $f$ of the form 
$$f(s)=s^2\int_0^\infty e^{-ts}g(t)dt$$ where $g$ is a Bernstein function. These 
are Bernstein functions and they are also called operator monotone functions. 
See \cite[Chapter 6]{SSV}.

Given a Bernstein function $f$, and a reversible Markov generator $A$, 
we can always form the operator $f(A)$ which is 
also the generator of a reversible Markov semigroup $e^{-tF(A)}$, $t\ge 0$. 
In the case of interest to us here, $A$ is the 
operator of right-convolution by $\delta_e-\phi$ on a group $G$ 
where $\phi$ is a symmetric probability measure which is (minus) the generator 
of the continuous time semigroup $e^{-tA}=H^\phi_t= \cdot *h^\phi_t$
with $h^\phi_t$ defined at (\ref{cont-time}).  Similarly, assuming 
$f(0)=a=0$ and $f(1)=1$, we have
$$f(A)= \cdot*(\delta_e-\phi_f)$$
with 
$$\phi_f= \sum_1^\infty c(f,n)\phi^{(n)}$$
where the coefficients $c(f,n)$ are given by the Taylor series 
$1-f(1-x)= \sum_1^\infty c(f,n)x^n$ at $x\sim 0$. Equivalently and 
more explicitly (see \cite{Bendikov2012}),
\begin{equation}\label{sub-cfn}
c(f,1)=b+\int_0^\infty te^{-t}\nu(dt),\;\;
c(f,n)=\int_0^\infty t^ne^{-t}\nu(dt), \; n>1.\end{equation}
 Obviously, 
the continuous time semigroup $e^{-tf(A)}$ is also the semigroup of 
right-convolution by $h^{\phi_f}_t$.  Further, because of 
the representation of $f$ using the measure $\nu$
(see the definition of Bernstein function), assuming that $a=0$, we have
$$f(A)= bA+ \int_{(0,\infty)}(I-e^{-tA})\nu(dt).$$ 
The following elegant result gives a sharp inequality for the Nash profile
of $f(A)$, that is, in our setting, the Nash profile of $\phi_f$.
\begin{theo}[{\cite[Theorem 1]{Schilling}}] \label{th-Sch} Let $f$ be a 
Bernstein function with $f(0)=0$, $f(1)=1$, and L\'evy measure $\nu$. 
Referring to the above 
setting and notation, for any symmetric probability measure $\phi$ on $G$, 
 the Nash profile ${\mathcal N}_{\phi_f}$ satisfies
$$\forall\, v>0,\;\; {\mathcal N}_{\phi_f}(v) \le \frac{2}{f(1/\mathcal N_\phi(2v))}.$$
Further, for any function $u$ such that $\|u\|_1^2/\|u\|_2^2\le v$,
\begin{equation}\label{key-Sch}
\frac{\mathcal E_{\phi_f}(u,u)}{\|u\|_2^2}\ge \frac{1}{2{\mathcal N}_\phi(2v)}
\int_0^{{\mathcal N}_\phi(2v)} \nu(s,\infty)ds.
\end{equation}
\end{theo}
The second statement is obtained in the proof of the first 
inequality provided in \cite{Schilling}.  
By Lemma \ref{lem-IsoNash}, the Nash profile inequality stated above 
translates into the 
$L^2$-isoperimetric profile inequality
\begin{equation}\label{Lambda-f}
 f (\Lambda_\phi(8v)/2)\le   2 \Lambda_{\phi_f}(v).
\end{equation}

\subsection{Extremal profile under a moment condition}

In this subsection, we focus on the $L^2$-profile 
$\Lambda_{2,\phi}=\Lambda_\phi$ and on 
symmetric probability measures $\phi$
with a finite weak moment $W(\rho,\phi)$ relative to a natural 
class of slowly varying functions $\rho$. We show that, in this context,
the upper bound of Theorem \ref{pro-compphiG} is sharp for any (amenable) 
group $G$. To make this important result precise we need 
the following notation.

Consider  the set of all continuous increasing
functions $\rho:[0,\infty)\ra [1,\infty)$ such that
\begin{equation}\label{ellrho}
\rho(t) \simeq \left(\int_{t}^\infty\frac{ds}{(1+s)\ell(s)}\right)^{-1}
\end{equation}
where $\ell$ is a continuous increasing regularly varying function 
$\ell:[0,\infty)\ra [1,\infty)$
of index $\alpha \ge 0$ and such that 
$\int_0^\infty\frac{ds}{(1+s)\ell(s)}<\infty$. 
Under the condition $\alpha\in [0,1)$,
\cite[Theorems 2.5-2.6]{Bendikov2012} shows that
$\rho(t) \asymp 1/\psi(1/t)$ where $\psi$ is a complete Bernstein function
satisfying $\psi(0)=0, \psi(1)=1$ and 
$\psi(s)\sim c\int_0^{s}\frac{ds}{s\ell(1/s)}$ for some $c>0$. Further, 
$1-\psi(1-x)=\sum_1^\infty c(\psi,n)x^n$ with $0\ge c(\psi,n)\sim 
\frac{1}{n\ell(n)}$.

Now, referring to (\ref{ellrho}), assume that $\alpha=0$ and 
that the slowly varying function $\ell$
satisfies $\ell(t^a)\simeq \ell(t)$ for any $a>0$.
Proposition 4.2 and Remark 4.4 of \cite{Bendikov2012} show that, 
on any group $G$, the symmetric probability measure 
$$\mathbf u_\psi=\sum_1^\infty c(\psi,n)\mathbf u^{(n)}$$
obtained by $\psi$-subordination of $\mathbf u$ (recall that $\mathbf u$ is 
uniform on the fixed generating set $S^*$ of $G$) satisfies
$$W(\rho,\mathbf u_\psi) \simeq \sup_{n\ge 1}\left\{\rho (n) \sum_n^\infty 
\frac{1}{k\ell(k)}\right\}  <+\infty.$$
That is, $\mathbf u_\psi$ has finite weak $\rho$-moment.

\begin{theo}\label{th-Lpsi}
Let $G$ be a finitely generated amenable group. 
Let $\rho:[0,\infty)\ra [1,\infty)$ be of the type {\em (\ref{ellrho})} with
$\ell$ slowly varying and satisfying $\ell(t^a)\simeq \ell(t)$ for all $a>0$.
Let $\psi$ be the associated complete Bernstein function described above. 
There are constants  $c=c(G,|S^*|,\rho), C=C(G,|S^*|,\rho)\in (0,\infty)$ 
such that $W(\rho,\mathbf u_\psi)\le C<$ and
$$  \frac{c}{\rho( 1/\Lambda_{\mathbf u}(v))}\le   
\Lambda_{\mathbf u_\psi}(v) \le \frac{C}{\rho( 1/\Lambda_{\mathbf u}(v))}.$$
Further, for any symmetric probability measure $\phi$ with $W(\rho,\phi)\le K$
$$\Lambda_{\phi}(v) \le \frac{CK}{\rho( 1/\Lambda_{\mathbf u}(v))}.$$
In particular, the extremal profile
$\widetilde{\boldsymbol \Lambda}_{G,\rho}$
satisfies
$$\widetilde{\boldsymbol \Lambda}_{G,\rho}(v)\simeq 
\frac{1}{\rho(1/\Lambda_{G}(v) )}.$$
\end{theo}
\begin{proof} Since the upper bounds are given by Theorem lower on $\Lambda_{\mathbf u_\psi}$.
This follows easily from (\ref{Lambda-f}), that is, from 
Lemma \ref{lem-IsoNash} and Theorem \ref{th-Sch} 
(i.e., \cite[Theorem 1]{Schilling}).  
\end{proof}

\begin{rem}Theorem \ref{th-Lpsi} 
can be interpreted as an ``almost positive'' answer to 
Problem \ref{prob1}(1) in the case where $\rho$ is of the type (\ref{ellrho})
with $\ell$ slowly varying and satisfying $\ell(t^a)\simeq \ell(t)$ 
for all $a>0$ (e.g., $\rho(t)=[1+\log (1+t)]^\alpha$, $\alpha>0$). Indeed, 
Theorem \ref{th-Lpsi} says that $\Lambda_G$ determines 
$\widetilde{\boldsymbol \Lambda}_{G,\rho}$ for such $\rho$ and this 
result can be transferred to $\Phi_G$ $\widetilde{\Phi}_{G,\rho}$ to the extend
that Theorems \ref{th-Coulhon1}-\ref{th-Coulhon2} give tight relations between 
the $\Lambda$'s and the $\Phi$'s.  See the next section for 
more explicit statements.
\end{rem}
\begin{rem}
On $\mathbb Z$ or $\mathbb Z^d$, if $\psi_\beta(s)=s^\beta$, $\beta\in (0,1)$,
$\rho_\alpha(s)=(1+s)^\alpha$, $\alpha>0$,
then $W(\rho_\alpha, \mathbf u_{\psi_{\alpha/2}})<\infty$ since, in fact,
$\mathbf u_{\psi_{\beta}}(x)\asymp (1+|x|)^{-\beta-d}$.  However, 
on a general amenable group $G$, it is not true that
$W(\rho_\alpha, \mathbf u_{\psi_{\alpha/2}})<\infty$. Indeed, the optimal 
moment condition one should expect from  $\mathbf u_{\psi_\beta}$ is a
weak $\rho_{\beta \gamma}$-moment were $\gamma\in [1/2,1]$ is 
the displacement exponent of simple random walk on $G$. 
See \cite{Bendikov2012} for details.  
Because of this, it is an open question whether $\Lambda_G$ determines
$\widetilde{\boldsymbol \Lambda}_{G,\rho_\alpha}$ for $\alpha\in (0,2)$ and, 
in fact, the authors believe the answer to this open question is likely 
to be negative.
\end{rem}

\section{Lower bounds on $\Phi_\rho$} \label{sec-lowPhi}
\setcounter{equation}{0}
Together, Lemma \ref{lem-phiLambda} and Theorem \ref{pro-compphiG}
provide an excellent way to obtain lower bounds on convolution powers 
of measures with a given moment condition, that is, on the group invariants  
$\Phi_{G,\rho},\widetilde{\Phi}_{G,\rho}$ of Definitions \ref{def1}-\ref{def2}.
This method is simpler than that of \cite{BSClmrw} and applies 
much more generally (the techniques developed in \cite{BSClmrw} 
provides additional inside and complementary results when they apply).

\begin{lem}[See, e.g., {\cite[Prosition 2.3]{CG}}] \label{lem-CG}
Referring to notation {(\ref{def-eig})}, there are constants 
$C,c\in (0,\infty)$ such that
for any symmetric probability measure 
$\mu$ on $G$, any finite subset $U\subset G$, and any $n=1,2,\dots$, we have
\begin{equation}
\mu^{(2n)}(e)\ge \frac{ce^{-Cn\lambda_\mu(U)}}{|U|}.
\end{equation}
\end{lem} 
\begin{proof}
Inspection indicate that $\lambda_\mu(U)$ 
is the lowest eigenvalue of the continuous time semigroup 
$$H^{U,\mu}_t=e^{-t}\sum_0^\infty \frac{t^n}{n!} K^n_{U,\mu}$$
where $K_{U,\mu}(x,y)=\mu(x^{-1}y)\mathbf 1_U(x)\mathbf 1_U(y)$.
Let $h^{U,\mu}_t(x,y)$ be the kernel of this semigroup, that is,
$$h^{U,\mu}_t(x,y)=e^{-t}\sum_0^\infty \frac{t^n}{n!} K^n_{U,\mu}(x,y).$$
By elementary spectral theory,  
$$e^{-t\lambda_\phi(U)}=
\sup\{\|H^U_tf\|_2: \mbox{support}(f)\subset U,\|f\|_2=1\}.$$
Note also that $h^{U,\mu}_t(x,y)\le h^\mu_t(x,y)$ for all $x,y \in U$.
Now, we have 
$$\mu^{(2n)}(e)\gtrsim h^\mu_{n}(e) \mbox{ and }
h^\mu_t(e)=\sup\{\|H_{t/2}^\mu f\|_2^2: \|f\|_1=1\}.$$
It follows that (see \cite[Proposition 2.3]{CG}), 
for any finite set $U$ and $f$ supported in $U$,
$$h^\mu_t(e)\ge  
\frac{\|f\|_2^2}{\|f\|_1^2}\frac{\|H^{U,\mu}_{t/2}f\|_2^2}{\|f\|_2^2}
\ge \frac{1}{|U|}\frac{\|H^{U,\mu}_{t/2}f\|_2^2}{\|f\|_2^2}
 .$$ 
Taking the supremum of all $f\neq 0$ with support in $U$, we obtain
that there are constants $c,C\in (0,\infty)$ such that for any finite set 
$U\subset G$ and any $n$,
$\mu^{(2n)}(e)\ge  \frac{ce^{-Cn\lambda_\mu(U)}}{|U|}.$ \end{proof}

\begin{theo} \label{th-rho-low}
Let $G$ be a finitely generated group equipped with a 
finite symmetric generating set and the associated word-length.
Let $\rho: [0,\infty)\ra [1,\infty)$ be an non-decreasing continuous function
and set  $M(t)=t^2/\int_0^t\frac{sds}{\rho(s)}$.
Let $\mu$ be a symmetric probability measure on $G$ satisfying the weak 
moment condition
$$W(\rho,\mu)=\sup_{s>0}\{s\mu(\{x:\rho(|x|)>s\})\}\le K.$$

\begin{itemize}
\item Assume that $\Lambda_{G}(v)\simeq v^{-2/D}$  (equivalently, 
$\Phi_G(n)\simeq n^{-D/2}$). Then
$$\mu^{(2n)}(e)\gtrsim \frac{1}{(M^{-1}(n))^D}$$
where $M^{-1}$ is the inverse function of $M$.
\item Let $\pi:[0,\infty)\ra [1,\infty)$ be an non-decreasing function such that $\pi(t)\le ct$ for some $c$ and assume 
that
$$\Phi_G(n)\ge \exp(-n/\pi(n)).$$
Then, there exist $a,A\in(0,\infty)$ 
such that for any integers $k,n$  we have
$$\mu^{(2n)}(e) \ge  \exp
\left( - A\left(\frac{k}{\pi(k)} +
\frac{n}{M(a\pi(k)^{1/2})}\right)\right).
$$
\end{itemize}
\end{theo}
\begin{proof} The first case follows straightforwardly from Lemma \ref{lem-CG},
Theorem \ref{pro-compphiG} and elementary computations. 
It is useful to note here that the first stated estimate 
is not sharp when $\rho$ is a slowly varying function. 
In this particular context 
(polynomial volume growth and $\rho$ slowly varying), 
the second stated result provides a sharp estimate. See the Corollary 
\ref{cor-Pol-Phi}.  Many of the results provided by this first case are 
already covered in \cite{BSClmrw,SCZ0,SCZ-pollog} by different methods 
but the case when $\rho$ is regularly varying of index $2$ is new.

In the second case, referring to
Lemma \ref{lem-phiLambda} applied to the measure $\mathbf u$, i.e., the 
uniform measure of the generating set $S$, for any natural integer $k$,  
let $U$ be a set of volume $\simeq $ to $\exp(Ak\pi(k))$. 
By Lemma \ref{lem-phiLambda}, we then have
$$ \Lambda_{\mathbf u}(\exp(Ak/\pi(k))\le \frac{A}{2\pi(k)}.$$
By Theorem \ref{pro-compphiG}, this gives
$$\Lambda_\mu(\exp(Ak/\pi(k))\le 
\frac{C(\mu,\rho,|S^*|)}{M(a\pi(k)^{1/2})}$$
for some constant $a>0$.
Putting these estimates together yields
$$\mu^{(2n)}(e)\gtrsim \exp\left( - \left( A'\frac{k}{\pi(k)}+
\frac{n}{M(a\pi(k)^{1/2})}\right)\right)$$
for some $A'\in (0,\infty)$.
\end{proof}

The following corollaries of Theorem \ref{th-rho-low} illustrate the 
wide applicability and the sharpness of the results this Theorem provides.
To state these results,  let us consider  the set of all continuous increasing
functions $\rho:[0,\infty)\ra [1,\infty)$  satisfying (\ref{ellrho}), that is, 
such that
$\rho(t) \simeq \left(\int_{t}^\infty\frac{ds}{(1+s)\ell(s)}\right)^{-1}$
where $\ell$ is a regularly varying function $\ell:[0,\infty)\ra [1,\infty)$
of index $\alpha \ge 0$ with
$\int_0^\infty\frac{ds}{(1+s)\ell(s)}<\infty$.  
Under this hypothesis the function $\rho$ is regularly varying (at infinity) of 
index $\alpha\in [0,\infty)$ and the probability measure
$$\phi_\ell(g)= c_\ell \sum_1^\infty \frac{1}{\ell (4^k)} \mathbf u_{4^k}$$ 
is well defined because  
$\sum \frac{1}{\ell(4^k)}\simeq \int _0^\infty \frac{ds}{s\ell(s)}$ and 
satisfies
\begin{eqnarray*}
W(\rho,\phi_\ell)&=&\sup_{s>0}\{s\phi_\ell(\{g: \rho(|g|)>s\})\}\\
&\le &
C \sup_k\left\{\rho(4^k) \int_{4^{k}}^\infty \frac{ds}{s\ell(s)}\right\}<+\infty.
\end{eqnarray*} 
This makes $\phi_\ell$ a potential witness for the behavior of
$\widetilde{\Phi}_{G,\rho}$. 

\begin{cor}\label{cor-Pol-Phi}
Let $G$ be a finitely generated group with polynomial volume growth 
of degree $D$. Let $\rho$ be as in {\em \ref{ellrho})} and set
$M(t)=t^2/\int_0^t\frac{sds}{\rho(s)}$.
\begin{enumerate}
\item Assume that $\alpha>0$. In this case
$$\widetilde{\Phi}_{G,\rho}(n) \simeq  1/(M^{-1}(n))^D .$$ 
\item Assume that $\rho$ is slowly varying and 
satisfies $$\rho\circ\exp (u)\simeq u^{1/\gamma} \kappa (u)$$ 
with $\gamma\in (0,\infty)$, $\kappa$ slowly varying at infinity and 
$\kappa(t^a)\simeq \kappa(t)$ for any $a>0$.
Then 
$$\widetilde{\Phi}_{G,\rho}(n)\simeq \exp\left(-  [n/\kappa(n)]^{\gamma/(1+\gamma)}\right).$$

\item Assume that the function  $\kappa = \rho \circ \exp$ is slowly varying 
and  satisfies  $s\kappa^{-1}(s)\simeq \kappa^{-1}(s)$ at infinity. 
Then
$$\widetilde{\Phi}_{G,\rho}(n)\simeq \exp\left(- n/
\kappa(n)\right).$$
\end{enumerate}
\end{cor}
\begin{proof}
For statement 1, the lower bound follows obviously 
from the first statement in Theorem \ref{th-rho-low}. The upper bound is 
provided by \cite[theorem 1.5]{SCZ0}.

The proofs of the last two statements are similar
and we give the details only  for statement (2). By the second statement in 
Theorem \ref{th-rho-low}, we have
$$\phi_{G,\rho}(n)\ge \exp\left(- A \left(\log k +n/\rho(k)\right)\right)$$
because the hypotheses on $\rho$ implies in particular that 
$\rho(k/\log k)\simeq \rho(k)$. Pick $k$ as a function of $n$ so that
$\log k \rho(k)\simeq n$. We then have
$\phi_{G,\rho}(n)\ge \exp(- C t)$ with   $t=\log k$
with $  t \rho\circ \exp(t) \simeq n$, that is, 
$t^{(1+\gamma)/\gamma}\kappa(t) \simeq n$. Because of the assumed property 
of $\kappa$, this yields 
$$t\simeq (n/\kappa(t))^{\gamma/(1+\gamma)} 
\simeq (n/\kappa(n))^{\gamma/(1+\gamma)}.$$
The matching upper bound can be derived from \cite[Theorem 2.4]{SCZ-pollog}
of by using the subordination results of \cite{Bendikov2012}.

\end{proof} 
\begin{exa} To illustrate case 1, consider the case when $\rho(s)=(1+s)^2$. 
Corollary \ref{cor-Pol-Phi} states implies that on a group with polynomial 
volume growth of degree $D$, any symmetric measure $\mu$ with 
finite second weak-moment  satisfies $\mu^{(2n)}(e)\gtrsim [n\log n]^{-D/2}.$
This was not known and could not be proved by the techniques 
of (\cite{BSClmrw}). In \cite{SCZ0}, the authors prove that the measure 
$\phi_2(x) =\frac{c}{(1+|x|)^{2+D}}$ 
(which has finite second weak-moment) satisfies
 $\phi_2^{(n)}(e)\simeq [n\log n]^{-D/2}$. Hence, $\phi_2$ provides a 
witness to the behavior of $\widetilde{\Phi}_{G,2}$.

The simplest illustration of case 2 is when $\rho(s)=(1+\log (1+s))^\alpha$.
In this case, the result reads
$$\widetilde{\Phi}_{G,\log_{[1]}^\alpha}(n)\simeq \exp(- n^{1/(1+\alpha}).$$
This was derived by a different method in \cite{SCZ-pollog}.

The last case, case 3, is illustrated by taking $\rho$ to be the power of an 
iterated logarithm, $\rho(s)=[\log_{[k]}(s)]^\alpha$, $k\ge 2$, $\alpha>0$, 
in which case we obtain
$$
\widetilde{\Phi}_{G,\log_{[k]}^\alpha}(n)\simeq \exp
\left(- n/[\log_{[k-1]}n]^{\alpha}\right).$$
This result was also derived by a different method in \cite{SCZ-pollog}.
\end{exa}

\begin{cor}\label{cor-Good-Phi} Assume that $G$ is a finitely generated 
group with exponential volume growth and such that 
$\Phi_{G}(n)\simeq \exp(-n^{1/3})$. Let $\rho$ be as in {\em (\ref{ellrho})}. 
\begin{enumerate}
\item Assume that $\rho$ is regularly varying of index $2$. In this case
$$\widetilde{\Phi}_{G,\rho}(n) \gtrsim 
\exp\left(- n^{1/3}\int_0^{n^{1/3}}\frac{sds}{\rho(s)}\right) .$$ 
\item Assume that $\rho(s)=(1+s)^\alpha$, $\alpha\in (0,2)$,
Then 
$$\widetilde{\Phi}_{G,\rho}(n))\simeq \exp\left( -n ^{1/(1+\alpha)}\right).$$
\item Assume that $\rho$  is slowly varying 
and  satisfies  $\rho(s^a)\simeq \rho(s)$ for any $a>0$. 
Then
$$\widetilde{\Phi}_{G,\rho}(n)\simeq \exp(- n/
\rho(n)).$$
\end{enumerate}
\end{cor}
\begin{proof} Each lower bound follows easily from Theorem \ref{th-rho-low}.
 In cases 2 and 3, the upper bound can be obtained by the simple method of 
\cite[Section 4.2]{BSClmrw}. In case 3, the upper bound can also be obtain 
by the subordination technique of \cite{Bendikov2012}.
\end{proof}
\begin{rem}We note that \cite{BSClmrw} contains a complete proof of 
both the upper and lower bound for case 2 but that it completely fails 
to cover the lower bound in cases 1 and 3. 
These lower bounds (cases 1 and 3) are 
new. Proving a matching upper bound in case 1 under the same hypotheses
is an interesting open question.  It is proved below that
the lower bound in case 1 is sharp in the case of the lamplighter group 
$\mathbb Z_2\wr \mathbb Z$. A matching upper bound for 
polycyclic groups of exponential growth will be given elsewhere.     
\end{rem}

\begin{cor}\label{cor-low-Phi1} Assume that $G$ is a finitely generated group
and that there exist $0<\gamma_1\le \gamma_2<1$ such that
$$\exp(-n^{\gamma_2})\lesssim \Phi_{G}(n)\lesssim \exp(-n ^{\gamma_1}).$$
Let $\rho$ be as in {\em \ref{ellrho})} and assume that $\rho$ is 
slowly varying function satisfying $\rho(t^a)\simeq \rho(t)$ for any $a>0$. 
Then
$$\widetilde{\Phi}_{G,\rho}(n) \simeq \exp(- n/\rho(n)).$$
\end{cor}
\begin{exa}For any  group in the large class described in Corollary
\ref{cor-low-Phi1}, we have
$$\widetilde{\Phi}_{G,\log_{[k]}^\alpha}(n)\simeq \exp(-n/[\log_{[k]} n]^\alpha) $$
for each $k=1,2,\dots$ and $\alpha>0$.
\end{exa}
\begin{proof}The lower bounds follows from Theorem \ref{th-rho-low} by inspection. The upper bound follows from the subordination technique in 
\cite{Bendikov2012}.
\end{proof}
The same proof gives the following complementary result.
\begin{cor}\label{cor-low-Phi2} Assume that $G$ is a finitely generated group
and that there exist continuous positive increasing functions of slow variation 
$\pi_1\le \pi_2$ such that $\pi_i(t^a)\simeq \pi_i(t)$ for all $a>0$ and 
$$\exp(-n/\pi_1(n))\lesssim \Phi_{G}(n)\lesssim \exp(-n/\pi_2(n) ).$$
Let $\rho$ be as in {\em \ref{ellrho})} and assume that $\rho$ is 
slowly varying function satisfying $\rho(t^a)\simeq \rho(t)$ for any $a>0$. 
Then
$$\exp(-n/\rho(\pi_1(n)))\lesssim
\widetilde{\Phi}_{G,\rho}(n) \lesssim \exp(- n/\rho(\pi_2(n))).$$
\end{cor}
\begin{exa} Let $\mathbf S_{d,r}$ be the free solvable group 
of solvable length $d$ on $r$ generators.  The behavior of 
$\Phi_{\mathbf S_{d,r}}$ is described in \cite{SCZ-fsol}. In particular, for $d>2$,
$$\Phi_{\mathbf S_{d,r}}(n)\simeq \exp\left(-n \left(\frac{\log_{[d-1]}n}{\log_{[d-2]}n}\right)^{2/r}\right).$$
For  $\rho$ as in (\ref{ellrho}), slowly varying and satisfying 
$\rho(t^a)\simeq \rho(t)$ for all $a>0$, we have 
$$\widetilde{\Phi}_{\mathbf S_{d,r},\rho}(n)\simeq \exp\left(-n/\rho(\log_{[d-2]}n)\right),\;\;d>2, r\ge 1.$$
\end{exa}
\begin{exa} Consider the iterated wreath products 
(the factor $\mathbb Z$ is repeated $k$ times)
$$W_k(\mathbb Z, \mathbb Z^d)=\mathbb Z\wr (\mathbb Z\wr (\dots    \mathbb Z\wr (\mathbb Z\wr \mathbb Z^d)\dots ))$$   
and
$$W^k(\mathbb Z,\mathbb Z^d)= 
(\dots ((\mathbb Z\wr \mathbb Z)\wr \dots)\wr \mathbb Z)\wr \mathbb Z^d.$$
From \cite{Erschler2006}, we know that
$$\Phi_{W_k(\mathbb Z,\mathbb Z^d)}(n)\simeq \exp\left(-n 
\left(\frac{\log_{[k]}n}{\log_{[k-1]}n}\right)^{2/d}\right) $$ 
and
$$\Phi_{W^k(\mathbb Z,\mathbb Z^d)}(n)\simeq \exp\left(-n^{(k+d)/(2+k+d)} 
\left(\log n\right)^{2/(2+k+d)}\right).$$  
If $\rho$ at (\ref{ellrho}) is slowly varying and satisfies 
$\rho(t^a)\simeq \rho(t)$ for all $a>0$, Corollaries 
\ref{cor-low-Phi1}--\ref{cor-low-Phi2} give
$$\Phi_{W_k(\mathbb Z,\mathbb Z^d),\rho}(n)\simeq 
\exp\left(-n /\rho(\log_{[k-1]}n)\right) $$ 
and
$$\Phi_{W^k(\mathbb Z,\mathbb Z^d),\rho}(n)\simeq \exp\left(-n/\rho(n)\right).$$  
\end{exa}

\section{Random walks on wreath products}\label{sec-wreath}
\setcounter{equation}{0}
This section is devoted to the computations of the behavior of a variety 
random walks on wreath products.

First we briefly review the definition of wreath products. 
Our notation follows \cite{Pittet2002} and \cite{SCZ-dv1}.
Let $H$, $K$ be two finitely generated groups. 
Denote the identity element of \ 
$K$ by $e_K$ and identity element of $H$ by $e_H$. Let $K_{H}$ denote the direct
sum:
\begin{equation*}
K_{H}=\sum_{h\in H}K_{h}.
\end{equation*}%
The elements of $K_{H}$ are functions $f:H\rightarrow K$, $h\mapsto f(h)=k_h$, 
which have finite support in the sense that  
$\{h\in H: f(h)=k_h\neq e_K\}$ is finite. Multiplication on $K_H$ is 
simply coordinate-wise multiplication. The identity element of $K_H$ is the
constant function $\boldsymbol e_K:h\mapsto e_K$ which, abusing notation, 
we denote by $e_K$.
The group $H$ acts on $K_{H}$ by left translation:%
\begin{equation*}
\tau _l(h)f(h')=f(h^{-1}h'),\;\;h,h'\in H.
\end{equation*}%
The wreath product $K\wr H$ is defined to be semidirect product%
\begin{equation*}
K\wr H=K_{H}\rtimes _{\tau }H,
\end{equation*}%
\begin{equation*}
(f,h)(f^{\prime },h^{\prime })=(f\cdot \tau _{l}(h)f^{\prime },hh^{\prime }).
\end{equation*}%
In the lamplighter interpretation of wreath products, $H$ corresponds
to the base on which the lamplighter lives and $K$ corresponds to the lamp. 
We embed $K$ and $H$
naturally in $K\wr H$ via the injective homomorphisms%
\begin{eqnarray*}
k &\longmapsto &\underline{k}=(\boldsymbol{k}_{e_H},e_H),
\;\;\boldsymbol k_{e_H}(e_H)= k,\; \boldsymbol k_{e_H}(h)=e_K \mbox{ if } h\neq e_H \\
h &\longmapsto &\underline{h}=(\boldsymbol  e_K,h).
\end{eqnarray*}%
Let $\mu_H $ and $\mu_K$ be probability measures on $H$ and $K$ respectively.
Through the embedding, $\mu_H $ and $\mu_K $ can be viewed as probability
measures on $K\wr H.$ Consider the measure 
\begin{equation*}
\nu=\mu_K \ast \mu_H \ast \mu_K 
\end{equation*}%
on $K\wr H$. This is called the switch-walk-switch measure on $K\wr H$ 
with switch-measure $\mu_K$ and walk-measure $\mu_H$.

We can also consider the measure (again, on $K\wr H$)
\begin{equation*}
\mu= \frac{1}{2}(\mu_K +\mu_H).
\end{equation*}%
We will mostly work with this type of measure which is better adapted to 
the techniques developed below.  We note that it is obvious that
$$\mathcal E_{p,\nu}\le C(\mu_H,\mu_K) \mathcal E_{p,\mu}.$$
Conversely, if $\mu_K,\mu_H$ are symmetric and $\mu_K(e_K)>0$, we also have
$$\mathcal E_{p,\nu}\le C'(\mu_H,\mu_K)\mathcal E_{p,\mu}.$$ 
So, for symmetric measures $\mu_H,\mu_K$ with $\mu_K(e)>0$, we have
$\Lambda_{p,\mu}\simeq \Lambda_{p.\nu}$
on $K\wr H$.

\subsection{Upper bounds for $\Lambda$ on wreath products}\label{sec-wreath1}
We describe a general upper bound on 
$\Lambda_{p,H_2\wr H_1,\mu}$ in terms of $\Lambda_{p,H_i,\mu_i}$, $i=1,2$ when 
$\mu=\frac{1}{2}(\mu_1+\mu_2)$.  Throughout this work,
$\Lambda^{-1}_{p,H,\mu}$ denotes the right-continuous  inverse of the 
non-decreasing function
$v\mapsto \Lambda_{p,H,\mu}(v)$.  

\begin{theo} \label{theo-wreathlow}
Let $\mu_i$ be a symmetric probability measures on $H_i$, $i=1,2$.
The measure $\mu=\frac{1}{2}(\mu_1+\mu_2)$ defined on $H_2\wr H_1$ satisfies
$$\Lambda_{p,H_2\wr H_1,  \mu}(v)\le s$$
for all $s,v>0$ such that 
$$v\ge \left(\Lambda^{-1}_{p,H_2,\mu_2}(s)\right)^{\Lambda^{-1}_{p,H_1,\mu_1}(s)}
\Lambda^{-1}_{p,H_1,\mu_1}(s)
$$
where 
$$\Lambda^{-1}_{p,H_i,\mu_i}(s)=\inf\{ v: \Lambda_{p,H_i,\mu_i}(v)\le s\}.$$
\end{theo}
\begin{proof} For each $s$ and $i=1,2$, let $v_i$ be the smallest $v$ such that
$\Lambda_{p,H_i,\mu_i} (v)\le s$. Let $\phi_i$ be a test function on $H_i$ 
such that  $|\mbox{support}(\phi_i)|\le v_i$ and
$$\frac{\mathcal E_{p,\mu_i}(\phi_i)}{\|\phi_i\|_p^p}=\Lambda_{p,H_i,\mu_i}(v_i).$$
Let $U_1$ be the support of $\phi_1$. Let $W$ be the set of functions 
$\eta:H_1\ra H_2$
whose support is contained in $U_1$ (i.e., $\eta(y)$ 
is equal to the identity element in $H_2$ when $y\not\in U_1$).
On $H_2\wr H_1$, consider the function 
$$ H_2\wr H_1\ni (\eta,x)\mapsto \phi(\eta,x)= \left(\prod_{y\in U_1} \phi_2(\eta(y))\right) \phi_1(x) \mathbf 1_{W}(\eta).$$
This function is supported on a set of size
$$|W||U_1|\le v_2^{v_1} v_1$$
and its $\ell^p$-norm on $H_2\wr H_1$
is given by
$$\|\phi\|_{\ell^p(H_2\wr H_1)}=\|\phi_1\|_{\ell^p(H_1)}^p\|\phi_2\|_{\ell^p(H_2)}^{p|U_1|}.$$
Next we have
\begin{eqnarray*}
\phi((\eta,x)(\mathbf e, z))-\phi((\eta,x)) &=&
\phi((\eta,xz))-\phi((\eta,x))\\
&=& \left(\prod_{y\in U_1} \phi_2(\eta(y))\right) 
(\phi_1(xz)-\phi_1(x)) \mathbf 1_{W}(\eta)
\end{eqnarray*}
and
\begin{eqnarray*}
\lefteqn{\phi((\eta,x)(\mathbf 1^{e_1}_t, e_1))-\phi((\eta,x))
=\phi((\eta \mathbf 1^x_t,x))-\phi((\eta,x))} \hspace{.5in}&&\\
&=& \left(\prod_{y\in U_1\setminus \{x\}} \phi_2(\eta(y))\right)
(\phi_2(\eta(x)t)-\phi_2(\eta(x)) 
\phi_1(x) \mathbf 1_{W}(\eta).
\end{eqnarray*}
This gives
\begin{eqnarray*}
\mathcal E_{p,\mu}(\phi)
&= &\frac{1}{2}\left(\frac{\mathcal E_{p,\mu_1}(\phi_1)}{\|\phi_1\|_{\ell^p(H_1)}^p}+
\frac{\mathcal E_{p,\mu_2}(\phi_2)}{\|\phi_2\|_{\ell^p(H_2)}^p}\right)
\|\phi_2\|_{\ell^p(H_2)}^{p|U_1|}\|\phi_1\|_{\ell^p(H_1)}^p\\
&\le &  s \|\phi\|_{\ell^p(H_1\wr H_2)}^p.
\end{eqnarray*}
This is the desired result.
\end{proof}

\subsection{Lower bounds on $\Lambda$ on wreath products}\label{sec-wreath2}

In \cite{Erschleriso,Erschler2006}, Erschler developed 
a method to bound the F\o lner functions of the wreath product $H_2\wr H_1$
from below in terms of the F\o lner functions of $H_1$ and $H_2$. 
This can be expressed as lower bounds on $\Lambda_{1,H_2\wr H_1}$ 
and yields good lower bounds on $\Lambda_{2,H_2\wr H_1}$ 
via the Cheeger inequality (\ref{Cheeger}). We generalize this in order 
to study spread-out measures on wreath product. 
Erschler \cite[Theorem 1]{Erschler2006}, which we recall below 
in a less general form, provides good lower bound for 
$\Lambda_{1,H_2\wr H_1,\mu}$ but, for spread-out measures, the 
Cheeger inequality might fail to provide good lower bound on 
$\Lambda_{2,H_2\wr H_1,\mu}$.  We combine comparison arguments  
with the results of Erschler to provide a widely applicable method 
to obtain satisfactory lower bounds on $\Lambda_{2,H_2\wr H_1,\mu}$.

Define the F\o lner function $\mbox{F\o l}_{G,\mu}$ by
$$\mbox{F\o l}(t)=\inf\{ v: \Lambda_{1,G,\mu}(v)\le  1/t\}.$$
Note that $\mbox{F\o l}_{G,\mu}=\Lambda_{1,G,\mu}^{-1}$ in the notation 
of Theorem \ref{theo-wreathlow}.
In the context of random walks on groups, 
\cite[Theorem 1]{Erschler2006} can be stated as follows.

\begin{theo}[{\cite[Theorem 1]{Erschler2006}}] \label{theo-Erschler}
There exists a constant $K\ge 1$ such that for any 
countable groups $H_i$ and symmetric probability measures $\mu_i$, $i=1,2$,
the measure $\mu=\frac{1}{2}(\mu_1+\mu_2)$ defined on $H_2\wr H_1$ satisfies
$$\Lambda_{1,H_2\wr H_1,  \mu}(v)\ge s/K$$
for all $s,v>0$ such that 
$$v\le \left(\Lambda^{-1}_{1,H_2,\mu_2}(s)\right)^{\Lambda^{-1}_{1,H_1,\mu_1}(s)/K}.
$$
\end{theo}

Consider the following problem. On a finitely generated group $G$, 
given a volume $v$, find a symmetric probability measure $\zeta_{G,v}$
such that $\Lambda_{1,G,\zeta_{G,v}}(v)\simeq 1$.  For instance, on any group $G$,
if we let $r(v)$ be the smallest radius of a ball of volume greater than $v$,
the uniform probability measure $\mathbf u_{r(2v)}$ on the ball of radius $r(2v)$ 
satisfies 
$$\Lambda_{1,G,\mathbf u_{r(2v)}}(v)\ge 1/2.$$

For our purpose we will need to consider the following question. Fix a 
symmetric probability measure $\mu$ and fix $t>0$. Given a solution 
$\zeta_{G,v}$ 
to the previous problem, what is the largest volume 
$v(t)$ such that $$ t \mathcal E_{G,\mu}\ge \mathcal E_{G,\zeta_{G,v(t)}}?$$
Solution to this problem can be obtained by using pseudo-Poincar\'e 
inequalities involving $\mathcal E_{G,\mu}$. For example, if $\mu=\mathbf u$ 
is the uniform measure on our generating set $S^*$, we have the 
pseudo-Poincar\'e inequality
$$\sum_{x}|f(xy)-f(x)|^2\le |S^*| |y|^2 \mathcal E_{G,\mathbf u}(f,f).$$
It follows that for a given $t$ we can choose  
$v(t)\simeq  V_G(\sqrt{t})$ to achieve
$$ t \mathcal E_{G,\mathbf u} \ge \mathcal E_{G,\mathbf u_{r(2v(t))}}.$$
The following proposition is based on this circle of ideas and 
is stated in a form that is suitable to treat iterated wreath products.
See Theorems \ref{th-longwreath} and  \ref{th-iterated} below.

\begin{pro}\label{pro-wreathup}
Let $\mu_i$ be a symmetric probability measures on $H_i$, $i=1,2$.
Fix $\delta>0$.
Assume that for each $t>0$ we can find $v^\delta_i(t)>0$  and symmetric 
probability measures $\zeta_{i,v_i}$ on $H_i$, $i=1,2$, such that
\begin{equation}
t\mathcal E_{H_i,\mu_i}\ge \mathcal E_{H_i,\zeta_{i,v^\delta_i(t)}} \mbox{ and }
\Lambda_{1,H_i,\zeta_{i,v^\delta_i(t)}}(v^\delta_i(t))\ge \delta.
\end{equation}
Then, for the measure $\mu=\frac{1}{2}(\mu_1+\mu_2)$ on $G=H_2\wr H_1$ 
and any $t>0$, we have
$$t\mathcal E_{G,\mu}\ge \mathcal E_{G,\zeta_{v(t)}} \mbox{ and }
\Lambda_{1,G,\zeta_{v(t)}}(v(t))\ge \delta/K$$ 
where $t\mapsto v(t)$ and the probability measure 
$\zeta_{v(t)}$ on $G$ are given by
$$v(t)= [v_2^\delta(t)]^{v_1^\delta(t)/K} \mbox{ and } \zeta_{v(t)}=\frac{1}{2}(\zeta_{1,v_1^\delta(t)}+\zeta_{2,v_2^\delta(t)}).$$
In particular,
$$\Lambda_{G,\mu}(v(t))\ge \frac{c}{t}\left(\frac{\delta}{K}\right)^2.$$
\end{pro}
\begin{proof} The hypotheses on $\mathcal E_{H_i,\mu_i}$ 
immediately imply that $t \mathcal E_{G,\mu}\ge \mathcal E_{G,\zeta_{v(t)}}$.
The lower bound on $\Lambda_{1,G,\zeta_{v(t)}}$ for the given volume $v(t)$ follows
from Erschler 's result stated in Theorem \ref{theo-Erschler}.
\end{proof}

This proposition will allow us to treat a great variety of examples.
To illustrate how this proposition work we treat two simple examples.

\begin{exa} On $\mathbb Z$, consider the measure $\mu_\alpha$ given by
$$\mu_\alpha(z)=c_\alpha(1+|z|)^{-\alpha-1}.$$
What is the behavior of $\mu^{(n)}(e)$ if 
$\mu=\frac{1}{2}(\mu_{\alpha_1} +\mu_{\alpha_2})$ on $\mathbb Z\wr \mathbb Z$
where $\mu_{\alpha_1}$ is supported on the base and $\mu_{\alpha_2}$ 
on the lamp above the identity of the base?

As noted in the Appendix section, on $\mathbb Z$ and for any $r>0$ we have
$$\mathcal E_{\mathbf u_r} \le C_\alpha r^\alpha \mathcal E_{\mu_\alpha}
\mbox{ and }   \Lambda_{1,\mathbb Z, \mathbf u_{r}}(r) \ge 1/2.$$ 
In other words, for any $t>0$ and $v_i(t)\simeq t^{1/\alpha_i}$, we have
$$\mathcal E_{\mathbf u_ {v_i(t)}} \le t \mathcal E_{\mu_{\alpha_i}}
\mbox{ and }   \Lambda_{1,\mathbb Z, \mathbf u_{v_i(t)}}(v_i(t)) \ge 1/2.$$ 
By Proposition \ref{pro-wreathup}, on $G=\mathbb Z\wr \mathbb Z$,
$$t \mathcal E_{G,\mu}\ge \mathcal E_{G, \zeta_{v(t)}} \mbox{ and }
\Lambda_{1,G,\zeta_{v(t)}}(v(t))\ge \frac{1}{2K}$$
where
$$v(t)\simeq \exp( t^{1/\alpha}\log t) \mbox{ and } \zeta_{G,v(t)}=
\frac{1}{2}(\mathbf u_{v_1(t)}+\mathbf u_{v_2(t)}).$$

It follows that
$$\Lambda_{2,\mathbb Z\wr \mathbb Z,\mu}(\exp(a t^{1/\alpha_1} \log t))\ge 1/t,$$
equivalently
$$\Lambda_{2,\mathbb Z\wr \mathbb Z,\mu}(v)\gtrsim \left(\frac{\log\log t}{\log t}
\right)^{\alpha_1}.$$
By Theorem \ref{th-Coulhon1}, this gives
$$\mu^{(n)}(e)\le \exp( - n^{\frac{1}{1+\alpha_1}}(\log n)^{\frac{\alpha_1}{1+\alpha_1}}).$$
Theorem \ref{theo-wreathlow}  
provides a matching lower bound (see also \cite{SCZ-dv1}). 

It is instructive to see what happens in this example if one applies 
directly Erschler F\o lner function results and Cheeger's inequality to obtain 
lower bound on $\Lambda_2$ and an upper bound on $\mu^{(n)}(e)$.
By Theorem \ref{th-A-pol-Lambda}, we know that at least for $\alpha \neq 1$
$$\Lambda_{1,\mathbb Z,\mu_\alpha}(v)\simeq \left\{
\begin{array}{cc} v^{-1}& \mbox{ if } 
\alpha\in (1,2)\\v^{-\alpha} & \mbox{ if } \alpha\in (0,1).\end{array}\right. $$
By \ref{theo-Erschler}, if $0<\alpha_1\neq 1$,
this implies (the value of $\alpha_2\in (0,2)$ 
does not matter)
$$\Lambda_{1,\mathbb Z\wr \mathbb Z,\mu}(v)\gtrsim \left\{
\begin{array}{cc} \frac{\log\log v}{\log v}& \mbox{ if } \alpha_1\in (1,2)\\
\left(\frac{\log\log v}{\log v}\right)^{\alpha_1} & \mbox{ if } 
\alpha_1\in (0,1).\end{array}\right. $$
In fact, these lower bounds admit  matching upper bounds.
Now, a lower bound on $\Lambda_{2,\mathbb Z\wr \mathbb Z,\mu}$ 
can be derived since
$\Lambda_{2,\mathbb Z\wr \mathbb Z,\mu} \gtrsim 
\Lambda_{1,\mathbb Z\wr \mathbb Z,\mu}^2.$ However this produces 
a lower bound that is significantly weaker than the one obtained 
above using Proposition \ref{pro-wreathup}.
\end{exa}

\begin{exa} Consider the wreath product 
$G=H\wr \mathbb Z$ where $H$ is a polycyclic group of 
exponential volume growth. On this group, we consider the measure
$\mu=\frac{1}{2}(\phi+\mathbf u)$ where $\phi$ is the measure on $\mathbb Z$ given
by $\phi(z)= c(1+|z|)^{-3}$ and $\mathbf u$ 
is the uniform measure on a finite symmetric generating set in $H$ 
containing the identity.  Recall that $\Phi_H(n)\simeq \exp(-n^{1/3})$,
$\Lambda_{2,H}(v)\simeq (\log(e+ v))^{-2}$ and 
$\Lambda_{1,H}(v)\simeq (\log (e+v))^{-1}$. Further,
$\Lambda_{1,H, \mathbf u_{r(2v)}}(v)\ge 1/2$ and $r(v)\simeq \log v $ 
since the volume function on $H$ has exponential growth. Also,
by the universal pseudo-Poincar\'e inequality for finitely supported symmetric
measure and associated word-length, we have
$$ Cr(2v)^2 \mathcal E_{H,\mathbf u_H}\ge \mathcal E_{H,\mathbf u_{r(2v)}}.$$

Next, note that the measure $\phi$ on $\mathbb Z$ sits in between the 
domains of attraction of symmetric stable law with parameter 
$\alpha\in (0,2)$ and the classical Gaussian domain of attraction. 
It is well-known (and it follows from 
Theorems \ref{th-Coulhon1}-\ref{th-A-pol-Lambda}) that 
 $\phi^{(n)}(0)\simeq (n\log n)^{-1/2}$,
$\Lambda_{2,\mathbb Z,\phi}(v)\simeq (1+v)^{-2}\log (e+v)$ and 
$\Lambda_{1,\mathbb Z,\phi}(v)\simeq (1+v)^{-1}$. We also have
$\Lambda_{1,\mathbb Z,\mathbf u_{r(2v)}}(v)\ge 1/2.$
Further, we have the 
pseudo-Poincar\'e inequality
$$\sum_{x\in \mathbb Z}|f(xy)-f(x)|^2\phi(y)\le 
C |y|^2(\log(e+|y|))^{-1}\mathcal E_{\mathbb Z,\phi}(f,f).$$

Applying Proposition \ref{pro-wreathup} with $H_1=\mathbb Z$, $\mu_1=\phi$, 
$H_2=H$, $\mu_2=\mathbf u$, the above data leads to 
$$v_1(t)\simeq (t\log t)^{-1/2},\;\;v_2(t)\simeq \exp( t^{1/2})$$
and
$$v(t)\simeq \exp(t(\log t)^{1/2}).$$
This gives
$$\Lambda_{2,G,\mu}(v)\gtrsim  \frac{(\log \log v)^{1/2}}{\log v}.$$
For comparison, we note that Theorem \ref{theo-Erschler}, gives
$$\Lambda_{1,G,\mu}(v)\gtrsim  \frac{1}{(\log v)^{1/2}}.$$
These two lower bounds can be complemented by matching upper bounds
using Theorem \ref{theo-wreathlow}.
\end{exa}

It is worth noting that Proposition \ref{pro-wreathup}  admits a 
version that leads to 
good lower bound for the $p$-isoperimetric profile $\Lambda_{p}$ 
on wreath products.  The proof is the same.

\begin{pro}\label{pro-pwreathup}
Let $\mu_i$ be a symmetric probability measures on $H_i$, $i=1,2$.
Fix $\delta>0$ and $p\in [1,\infty)$.
Assume that for each $t>0$ we can find $v^\delta_i(t)>0$  and symmetric 
probability measures $\zeta_{i,v_i}$ on $H_i$, $i=1,2$, such that
\begin{equation}
t\mathcal E_{p,H_i,\mu_i}\ge \mathcal E_{p,H_i,\zeta_{i,v^\delta_i(t)}} \mbox{ and }
\Lambda_{1,H_i,\zeta_{i,v^\delta_i(t)}}(v^\delta_i(t))\ge \delta.
\end{equation}
Then, for the measure $\mu=\frac{1}{2}(\mu_1+\mu_2)$ on $G=H_2\wr H_1$ 
and any $t>0$, we have
$$t\mathcal E_{p,G,\mu}\ge \mathcal E_{p,G,\zeta_{v(t)}} \mbox{ and }
\Lambda_{1,G,\zeta_{v(t)}}(v(t))\ge \delta/K$$ 
where $t\mapsto v(t)$ and the probability measure 
$\zeta_{v(t)}$ on $G$ are given by
$$v(t)= [v_2^\delta(t)]^{v_1^\delta(t)/K} \mbox{ and } \zeta_{v(t)}=\frac{1}{2}(\zeta_{1,v_1^\delta(t)}+\zeta_{2,v_2^\delta(t)}).$$
\end{pro}

We now state a theorem that uses the iterative nature of Proposition 
\ref{pro-wreathup}. Consider a  sequence $(H_i)_1^m$ of finitely 
generated groups. Since taking wreath product is neither commutative nor 
associative, this sequence gives rise to many different iterated wreath 
product including $H_m\wr (H_{m-1}\wr (\cdots (H_2\wr H_1)\cdots))$
and $(\cdots (H_m\wr H_{m-1})\wr \cdots H_2)\wr H_1$. Let $\mathfrak B$
be a symbol of length $m$ describing a possible bracketing and
$W_\mathfrak B (H_m,\dots,H_1)$ be the corresponding wreath product. 
This can be define inductively
with $(\cdot)$ representing the bracketing of one single group, 
$( \wr )$, representing the bracketing of groups 
(i.e.   gives $H_2\wr H_1$). Inductively, if $\mathfrak B_1, \mathfrak B_2$ 
are such symbols, then $\mathfrak B=(\mathfrak B_2\wr \mathfrak B_1)$ is
also such a symbol and 
$$W_\mathfrak B(H_m,\dots H_1)=W_{\mathfrak B_2}(H_m,\dots, H_{m_1+1})\wr 
W_{\mathfrak B_1}(H_{m_1},\dots, H_1).$$
Note that the length of $\mathfrak B$ is defined inductively as the sum of 
the length of $\mathfrak B_1,\mathfrak B_2$ and length of $(\cdot)$ equal $1$.
We can now introduce a similar operation on sequences of numbers 
$(v_1,\dots, v_m)$ by setting 
$$W_{(\cdot)}(v)=v,\;W_{(\cdot\wr \cdot)}(v_2,v_1)= v_2^{v_1/K} $$ and, 
if $\mathfrak B= (\mathfrak B_2\wr \mathfrak B_1)$ as above,
$$W_{\mathfrak B}(v_m,\dots, v_1)=
W_{\mathfrak B_2}(v_m,\dots, v_{m_1+1})^{
W_{\mathfrak B_1}(v_{m_1},\dots, v_1)/K}.
$$
Here $K$ is the constant provided by Erschler's theorem, i.e., Theorem 
\ref{theo-Erschler}.

Similarly, given probability measures $\mu_i$ on $H_i$, $1\le i\le m$,
and $\mathfrak B=(\mathfrak B_2\wr \mathfrak B_1)$
define $\mu_\mathfrak B$ to be the probability measure on 
$W_{\mathfrak B}(H_m,\dots,H_1)$ define inductively by
$$\mu_{\mathfrak B}=\frac{1}{2}\left(\mu_{\mathfrak B_2}+\mu_{\mathfrak B_1}\right)
$$ 
where $\mu_{\mathfrak B_2}$ is understood as a probability measure
on $W_{\mathfrak B}(H_m,\dots,H_1)$ supported on the copy of 
$W_{\mathfrak B_2}(H_m,\dots,H_{m_1+1})$ above the identity element of 
$W_{\mathfrak B_1}(H_{m_1},\dots,H_1)$ and $\mu_{\mathfrak B_2}$ is
a probability measure on $W_{\mathfrak B}(H_m,\dots,H_1)$ supported on  
$W_{\mathfrak B_1}(H_{m_1},\dots,H_1)$.  For instance, given $\mu_i, H_i$, 
$1\le i\le 3$ and $\mathfrak B=((\cdot\wr\cdot)\wr (\cdot))$,
$\mu_{\mathfrak B}$ is a measure on $(H_3\wr H_2)\wr H_1$ and is equal to
$$\mu_{\mathfrak B}=\frac{1}{2}\left( \frac{1}{2}(\mu_3+\mu_2) +\mu_1\right)$$ 
where $\frac{1}{2}(\mu_3+\mu_2)$ is the measure on $H_3\wr H_2$ 
considered in Theorem \ref{th-wreathlow} and Proposition \ref{pro-wreathup}.
In some instance, it is useful to write
\begin{equation}\label{muB}
\mu_{\mathfrak B}=\nu_{\mathfrak B,\mu_m,\dots,\mu_1}
\end{equation}
to specify the measures used in the construction.

\begin{theo}\label{th-longwreath}
Let $\mu_i$ be a symmetric probability measures on $H_i$, $1\le i\le m$.
Fix $\delta>0$ and $p\ge 1$. 
assume that for each $t>0$ we can find $v^\delta_i(t)>0$  and symmetric 
probability measures $\zeta_{i,v_i}$ on $H_i$, $1\le i \le m$, such that
\begin{equation}
t\mathcal E_{p,H_i,\mu_i}\ge \mathcal E_{p,H_i,\zeta_{i,v^\delta_i(t)}} \mbox{ and }
\Lambda_{1,H_i,\zeta_{i,v^\delta_i(t)}}(v^\delta_i(t))\ge \delta.
\end{equation}
Fix a symbol $\mathfrak B$ of length $m$ as above. Then, for any $t>0$,  
the measure $\mu=\mu_\mathfrak B$ on $G=W_{\mathfrak B}(H_m,\dots,H_1)$ 
satisfies
$$t\mathcal E_{p,G,\mu}\ge \mathcal E_{p,G,\zeta_{v(t)}} \mbox{ and }
\Lambda_{1,G,\zeta_{v(t)}}(v(t))\ge \delta/K^m$$ 
where $t\mapsto v(t)$ and the probability measure 
$\zeta_{v(t)}$ on $G$ are given by
$$v(t)= W_{\mathfrak B}(v^\delta_m(t),\dots,v^\delta_1(t)) \mbox{ and } 
\zeta_{v(t)}=\nu_{\mathfrak B,\zeta_{m,v_m^\delta(t)},\dots,
\zeta_{1,v_1^\delta(t)}}.$$
In particular,
$$\Lambda_{p,G,\mu}(v(t))\ge \frac{c(1,p)}{t}\left(\frac{\delta}{K^m}\right)^p.$$
\end{theo}

\begin{exa} Let $p=2$. Assume that $H_i$ 
is a group of polynomial volume growth of degree $d_i$, $1\le i\le 4$.
On $H_i$, consider the measure 
$\mu_i(h)\asymp (1+|h|)^{-\alpha_i-d_i}$, $\alpha_i\in (0,2)$, $1\le i\le 4$.
The symbols $\mathfrak B$ of length four are $\mathfrak B_1=(((\cdot\wr \cdot)\wr \cdot)\wr\cdot)$, $\mathfrak B_2=((\cdot\wr (\cdot\wr \cdot))\wr\cdot)$, 
$\mathfrak B_3=((\cdot\wr \cdot)\wr (\cdot\wr\cdot))$,
$\mathfrak B_4=(\cdot\wr ((\cdot\wr \cdot)\wr \cdot))$ 
and  $\mathfrak B_5=(\cdot\wr ( \cdot\wr (\cdot\wr\cdot)))$.
Set $v_{\mathfrak B}=W_\mathfrak B(v_4,\dots,v_1)$. By inspection, we have
$$v_{\mathfrak B_1}(t)\simeq \exp\left(t^{\frac{d_1}{\alpha_1}+\frac{d_2}{\alpha_2}+\frac{d_3}{\alpha_3}}
\log t\right),\;\;
v_{\mathfrak B_2}(t)\simeq \exp_{[2]}\left(t^\frac{d_2}{\alpha_2}\log t\right),$$
$$v_{\mathfrak B_3}(t)\simeq \exp_{[2]}\left(t^{\frac{d_1}{\alpha_1}}\log t\right),\;\;
v_{\mathfrak B_4}(t)\simeq \exp_{[2]}\left(t^{\frac{d_1}{\alpha_1}+\frac{d_2}{\alpha_2}}
\log t\right),$$
and 
$$
v_{\mathfrak B_5}(t)\simeq \exp_{[3]}\left(t^{\frac{d_1}{\alpha_1}}\log t\right).$$
This gives
$$\Lambda_{W_{\mathfrak B_1},\mu_{\mathfrak B_1}}(v) \simeq 
\left(\frac{ \log \log v}{\log v}\right)^{1/(\frac{d_2}{\alpha_2}+\frac{d_3}{\alpha_3})},$$
$$\Lambda_{W_{\mathfrak B_2},\mu_{\mathfrak B_2}}(v) \simeq 
\left(\frac{ \log \log\log  v}{\log\log v}\right)^{\alpha_2/d_2},\;\;
\Lambda_{W_{\mathfrak B_3},\mu_{\mathfrak B_3}}(v) \simeq 
\left(\frac{ \log \log\log  v}{\log\log v}\right)^{\alpha_1/d_1},$$
$$\Lambda_{W_{\mathfrak B_4},\mu_{\mathfrak B_4}}(v) \simeq 
\left(\frac{ \log \log\log  v}{\log\log v}\right)^{1/(\frac{d_1}{\alpha_1}+\frac{d_2}{\alpha_2})},$$
and

$$\Lambda_{W_{\mathfrak B_5},\mu_{\mathfrak B_5}}(v) \simeq 
\left(\frac{ \log\log\log\log  v}{\log\log\log v}\right)^{\alpha_1/d_1}.$$

\end{exa}

\subsection{Comparison measures and applications} 
The main theorems stated in the previous sections require that, 
for any symmetric probability measure $\mu$, we exhibit a collection 
$\zeta_v$,  $v>0$, 
of spread-out symmetric probability measures with the property that
\begin{equation}\label{delta}
\Lambda_{1,\zeta_v}(v)\ge \delta \end{equation}
for some fixed $\delta\in (0,1)$ and such that we can control
$\mathcal E_{\zeta_v}$ in terms of $v$ and  $\mathcal E_\phi$. 
The following two theorems show that we can always produce 
such a collection of measures.  

The first of these two theorems apply to subordinated measures $\phi_f$. 
Namely, given a Bernstein function with L\'evy 
measure $\nu$ and $t>0$, set 
$$\nu_t(ds)= [\nu((t,\infty))]^{-1}
\mathbf 1_{(t,\infty)}\nu(ds).$$ That is, the measures $\nu_t$, $t>0$,
are the normalized tail measures of $\nu$. Let 
\begin{equation} \label{cfnt}
c(t,f,n)=c(f_t,n)
\end{equation}
be the coefficients associated by (\ref{sub-cfn}) with the Bernstein function 
$$f_t(s)= \left(\int_{(0,\infty)}e^{-\tau}\nu_t(d\tau)\right)s+ 
\int_{(0,\infty)}(1-e^{-s\tau}) \nu_t(d\tau).$$ 
Note that $f_t(0)=0$, $f_t(1)=1$.

\begin{theo}[spread-out measures for subordinated measures]
Let $\phi$ be a symmetric probability measure on a countable 
group $G$ with Nash profile ${\mathcal N}_\phi$. 
Let $f$ be a Bernstein function with L\'evy measure $\nu$ and 
such that $f(0)=0$, $f(1)=1$. Then
the measures $\zeta^f_v =\phi_{f_t}$, $t={\mathcal N}_\phi(2v)$, $v>0$,  
satisfy 
\begin{equation}\label{f-large}
\Lambda_{1,\zeta^f_v}(v)\ge \Lambda_{2,\zeta^f_v}(v)\ge \frac{1}{2}\;\; \mbox{ and } 
\;\;\mathcal E_{\phi_f}\ge \nu\left(({\mathcal N}_\phi(2v),\infty)\right) 
\mathcal E_{\zeta^f_v}.
\end{equation}
\end{theo}
\begin{exa} Assume that $f(s)=bs+\int_{(),\infty)}(1-e^{-ts})\nu(ds)$ 
is a complete Bernstein  with $\nu(ds)=g(s)ds$, $f(1)=1$ and
$f'(s)-b \sim \frac{s^\alpha}{s\ell(1/s)}$ at $0^+$ where $\ell $ 
is slowly varying at infinity and $\alpha\in [0,1]$.  
By \cite[(2.13)]{Bendikov2012}, we have
$g(s)\sim 1/\left[\Gamma(1-\alpha) s^{1+\alpha}\ell(s)\right]$ at infinity which 
implies $\nu((s,\infty))\sim c_\alpha f(1/s)$. This and Lemma \ref{lem-IsoNash}
means that (\ref{f-large})
is equivalent to the more explicit statement
$$ \Lambda_{1,\zeta^f_v}(v)\ge \Lambda_{2,\zeta^f_v}(v)\ge \frac{1}{2}\;\; \mbox{ and } 
\;\;\mathcal E_{\phi_f}\ge c_\alpha f(\Lambda_\phi(8v)/2)) 
\mathcal E_{\zeta^f_v}.$$
\end{exa}

\begin{proof} Write $f(s)= s +\int_{(0,\infty)}(1-e^{-st})\nu(dt)$.
First applies (\ref{key-Sch}) to $f_t$. Since, by definition,  
$\zeta^f_v=\phi_{f_t}$  and
$\nu_t(I)=0$ if the interval $I$ is contained in $(0,\mathcal N_\phi(2v))$, 
for any function $u$ with finite support, we have
$$\mathcal E_{\zeta^f_v}(u,)\ge \frac{1}{2} \|u\|_2^2.$$
That is $\Lambda_{\zeta^f_v}(v) \ge \frac{1}{2}$.

Next, write $A=\cdot*(\delta_e-\phi)$ (right-convolution  by $\delta_e-\phi$)
and
$$\mathcal E_{\zeta^f_v}(u,u) = 
\int_0^\infty \tau e^{-\tau}\nu_t(d\tau)
\mathcal E_\phi(u,u)
+\int_{(0,\infty)}\langle (I-e^{-\tau A})u,u\rangle \nu_t(d\tau)$$
Since 
$$\int_0^\infty \tau e^{-\tau}\nu_t(d\tau)\le 
\frac{1}{\nu((\mathcal N_\phi(2v),\infty))}
\int_0^\infty \tau e^{-\tau}\nu(d\tau), $$
and 
$$\int_{(0,\infty)}\langle (I-e^{-\tau A})u,u\rangle \nu_t(d\tau)
\le  \frac{1}{\nu((\mathcal N_\phi(2v),\infty))}
\int_{(0,\infty)}\langle (I-e^{-\tau A})u,u\rangle \nu(d\tau)$$
it follows that
$$\mathcal E_{\zeta^f_v}(u,u) \le 
\frac{1}{\nu((\mathcal N_\phi(2v),\infty))}\mathcal E_{\phi_f}(u,u).$$
\end{proof}

\begin{exa} Consider the Bernstein function $f_\alpha(s)=s^{\alpha}$,
$\alpha\in(0,1]$. In this case $\nu_\alpha(dt)=
\frac{\alpha}{\Gamma(1-\alpha)}t^{-\alpha-1}$. The construction above yields spread-out measures $\{\zeta^{f_\alpha}_{v}\}$ such that 
$$\Lambda_{2,\zeta^{f_\alpha}_v}(v)\ge \frac{1}{2}$$
and (using Lemma \ref{lem-IsoNash} for the last inequality)
$$
\mathcal{E}_{\phi_{f_\alpha}}\ge\frac{\mathcal N_\phi(2v)^{-\alpha}}{\Gamma(1-\alpha)}\mathcal{E}_{\zeta_{v}}\geq c_\alpha 
\Lambda_{2,\phi}(8v)^{\alpha}\mathcal{E}_{\zeta_{v}}.$$
\end{exa}

The next result apply to any symmetric probability measure $\phi$. 
For any fixed $\alpha\in (0,1)$ and $t>0$, consider the L\'evy measure
$$\nu^\alpha_t(ds)= \kappa(\alpha,t) \mathbf 1_{(t,2t)}(s) s^{-\alpha-1}ds, 
\;\;\kappa(\alpha,t)
= \alpha (1-2^{-\alpha})^{-1} t^\alpha$$
and
$b^\alpha_t= \kappa (\alpha,t) \int_t^{2t}e^{-u} u^{-\alpha-1}du.$  
Denote by $f^\alpha_t$ the Bernstein function $$f^\alpha_t(s)
=b^\alpha_t s +\int_{(0,\infty)}(1-e^{-su})\nu^\alpha_t(du)$$
and note that, by construction, $f^\alpha_t(0)=0$, $f^\alpha_t(1)=1$.
The Bernstein function $f^\alpha_t$ is a localized version of the classical 
Bernstein function $s\mapsto s^\alpha, \alpha\in (0,1)$. For the applications we have in mind, using any arbitrary fixed value of 
$\alpha\in 0,1)$ in the following theorem will be adequate.  
\begin{theo}[spread-out measures for $\phi$] \label{th-subtrick}
Let $\phi$ be a symmetric probability measure on a countable 
group $G$ with $L^2$-isoperimetric profile $\Lambda_\phi$. 
Fix $\alpha\in (0,1)$. Then
the measures $\zeta^\alpha_v =\phi_{f^\alpha_t}$, $t={\mathcal N}_\phi(2v)$, $v>0$,  satisfy 
\begin{equation}\label{phi-large}
\Lambda_{1,\zeta^\alpha_v}(v)\ge \Lambda_{2,\zeta^\alpha_v}(v)\ge \frac{1}{2}\;\; 
\mbox{ and } 
\mathcal E_{\phi}\ge c_\alpha \Lambda_\phi(8v) 
\mathcal E_{\zeta^\alpha_v}.
\end{equation}
\end{theo}
\begin{proof} 
The proof of the first inequality in (\ref{phi-large}) 
is the same as in the case of (\ref{f-large}). For the 
Dirichlet form comparison, write again $A=\cdot*(\delta_e-\phi)$ 
and recall that $t^{-1}\langle (I-e^{-tA})u,u\rangle$
is an increasing function of $t$ with limit $\mathcal E_\phi(u,u)$. 
It follows that
\begin{eqnarray*}\mathcal E_{\zeta^\alpha_v}(u,u) &=&
b^\alpha_t \mathcal E_\phi(u,u)
+\int_{(0,\infty)}\langle (I-e^{-\tau A})u,u\rangle \nu^\alpha_t(d\tau)\\
&\le& \left(
b^\alpha_t 
+\int_{(0,\infty)}\tau \nu^\alpha_t(d\tau)\right)\mathcal E_\phi(u,u). 
\end{eqnarray*}
Since $b^\alpha_t\le e^{-t}$ and $\int_0^\infty \tau \nu^\alpha_t(d\tau)
= \frac{\alpha (2^{1-\alpha})}{(1-\alpha)(1-2^{-\alpha})}t$ and $t\ge 1$, 
this gives 
$$\mathcal E_{\zeta^\alpha_v} \le c_\alpha t \mathcal E_\phi.$$
The desired result (with a different $c_\alpha$) 
follows since, by Lemma \ref{lem-IsoNash}, 
$$t=\mathcal N_\phi(2v)\le 2/\Lambda_\phi(8v).$$
\end{proof}

Theorem \ref{th-subtrick} turns the statements of Section \ref{sec-wreath2} 
into very effective results by providing the needed hypotheses. In particular, Theorem \ref{th-WR} stated in the introduction follows 
immediately from Proposition \ref{pro-wreathup} and 
Theorem \ref{th-subtrick}. 
Similarly, the case $p=2$ of Theorem \ref{th-longwreath} and Theorem
\ref{th-subtrick} yields 
the following statement.

\begin{theo}\label{th-longwreath2}
Let $\mu_i$ be a symmetric probability measures on $H_i$, $1\le i\le m$.
Fix a symbol $\mathfrak B$ of length $m$ as in {\em Theorem
\ref{th-longwreath}}.  Then, for any $v,s>0$,  
the measure $\mu=\mu_\mathfrak B$ on $G=W_{\mathfrak B}(H_m,\dots,H_1)$ 
satisfies
$$\Lambda_{2,G,\mu}(v)\ge s/K^m \mbox{ for any }
v \le W_{\mathfrak B}(\Lambda^{-1}_{2,H_1,\mu_1}(s),
\dots,\Lambda^{-1}_{2,H_m,\mu_m}(s)).$$
\end{theo}

\section{Spread-out random walks on wreath products} \label{sec-wreathexa}
\setcounter{equation}{0}

This section provides a host of explicit examples where the behavior
of random walks associated with spread-out measures on wreath products 
can be computed. In particular, we obtain a variety of 
sharp estimates for $\Phi_{G,\rho}$
when $G$ is a wreath product (or an iterated wreath product) and $\rho$ is 
a moment function. 

\subsection{Groups where $\Lambda_G$ is controlled by volume growth}

We say that $\Lambda_G$ is controlled by volume growth if 
$\Lambda_G\simeq \mathcal W_G^{-2}$ where $\mathcal W_G(v)=\inf\{r: V_G(r)>v\}$. 
It is always true that  $\Lambda_G(v)\gtrsim \mathcal  W_G^{-2}$ 
(this follows from the $L^2$-version of the argument in \cite{CSCiso}, 
see the appendix for variations). Groups quasi-isometric to
polycyclic groups satisfy $\Lambda_G\simeq \mathcal W_G^{-2}$ and  
Tessera \cite[Theorem 4]{Tessera}
describes a large class of groups of exponential volume growth 
(Geometrically Elementary Solvable or GES groups) which satisfy 
$\Lambda_G\simeq \mathcal  W_G^{-2}$. In all these cases the volume growth function 
is of type $V_G(r)\simeq r^d$ or $V_G(r)\simeq \exp(r)$ and
$\Lambda_G(r)\simeq v^{-2/d}$ (equivalently $\Phi_G(n)\simeq n^{-d/2}$) 
or $\Lambda_G(v) \simeq (\log(1+ v))^{-2}$ (equivalently 
$\Phi_G(n)\simeq \exp(-n^{1/3})$), respectively. 

\begin{theo} \label{th-iterated}
Let $(H_i)$, $1\le i\le m$ be groups for which $\Lambda_{H_i}
\simeq \mathcal W_{H_i}^{-2}$. For each $i$, $1\le i\le m$, let 
$\mu_i(h)=c_i\sum_1^\infty 4^{\alpha_i k}\mathbf u_{H_i}(4^k)$ 
with $\alpha_i\in (0,2)$.  Referring to the notation of {\em Theorem 
\ref{th-longwreath}}, fix a wreath product symbol $\mathfrak B$ of length $m$
and consider the measure $\mu_{\mathfrak B}=\nu_{\mathfrak B,\mu_m,\dots,\mu_1}$
defined at {\em (\ref{muB})} on the wreath product 
$W_{\mathfrak B}(H_m,\dots,H_1)$.  Then the $\simeq$-class of 
$\Lambda_{\mu_{\mathfrak B}}$ can be computed and is described by 
{\em Theorem \ref{th-longwreath2}}. In particular, when $m=2$ and 
$\mu=\mu_{(\cdot\wr\cdot)}$ on $H_2\wr H_1$,
\begin{itemize}
\item If $V_{H_1}$ is exponential and $H_2$ non-trivial, 
$\mu^{(2n)}(e)\simeq \exp(-n/[\log n]^{\alpha_1}) $
\item If $V_{H_1}(r)\simeq r^{d_1}$ and $V_{H_2}(r)\simeq r^{d_2}$,
$$\mu^{(2n)}(e)\simeq \exp\left( 
-n^{d_1/(\alpha_1+d_1)}[\log n]^{\alpha_1/(\alpha_1+d_1)}\right).$$
\item If $V_{H_1}(r)\simeq r^{d_1}$ and $V_{H_2}(r)\simeq \exp(r)$,
$$\mu^{(2n)}(e)\simeq \exp\left( 
-n^{(\alpha_1+\alpha_2d_1)/(\alpha_1+\alpha_2d_1+\alpha_1\alpha_2)}\right).$$
\end{itemize}
\end{theo}

Let $K$ be a finitely generated group which will be either finite,
of polynomial volume growth or of exponential volume growth and such that 
$\Phi_{K}(n)\simeq \exp(-n^{1/3})$. For instance, $K$ could be any polycyclic 
group. Let $H$ be a group of polynomial volume growth. 
In any of these cases, $\Phi_{K\wr H}$ is known 
(thanks to the results of \cite{PSCwp,Erschler2006}). 
In the first case ($K$ finite) $\Phi_{K\wr H}(n)\simeq \exp(-n^{d/(2+d)})$.
In the second case, $\Phi_{K,\wr H}(n)\simeq \exp(-n^{d/(2+d)}(\log n)^{2/(2+d)})$
and in the third case, $\Phi_{K,\wr H}(n)\simeq \exp(-n^{(1+d)/(3+d)})$.
In particular, Corollary \ref{cor-low-Phi1} applies to 
these groups and gives that for any slowly varying function $\rho$ 
as in (\ref{ellrho}) such that $\rho(t^a)\simeq \rho(t)$ for each $a>0$, 
we have  
$$\Phi_{K\wr H,\rho}(n)\simeq \exp(-n/\rho (n)).$$

The following two theorems provide the behavior of 
$\widetilde{\Phi}_{K\wr H, \rho}$ for  $\rho(s)=\rho_\alpha(s)=(1+s)^\alpha$, 
$\alpha \in (0,2)$ and for $\rho(s)$ regularly varying of index $2$. 
\begin{theo} \label{th-wr-Pol-alpha}
Let $H$ be a group of polynomial volume growth of degree $d$. 
\begin{enumerate}
\item If $K\neq \{e_K\}$ is finite, we have
$$\widetilde{\Phi}_{K\wr H,\rho_\alpha}(n)\simeq 
\exp\left(-n^{d/(\alpha+d)}\right).$$
\item If $K$ is not finite and has polynomial volume growth, we have
$$\widetilde{\Phi}_{K\wr H,\rho_\alpha}(n)\simeq 
\exp\left(-n^{d/(\alpha+d)} (\log n)^{\alpha/(\alpha+d)}\right)  .$$
\item If $K$ has exponential growth and satisfies $\Phi_{K}(n)\simeq 
\exp(-n^{1/3})$, we have
$$\widetilde{\Phi}_{K\wr H,\rho_\alpha}(n)\simeq 
\exp\left(-n^{(d+1)/(\alpha+d+1)}\right).$$
\end{enumerate}
\end{theo}
\begin{proof}
The lower bounds are already derived in \cite{BSClmrw}. They also follow from 
Theorem \ref{th-rho-low}. The upper bounds follow from 
Theorem \ref{th-WR} and known results on $K,H$. Consider 
for instance 
the case when $K$ has exponential volume growth. To obtain an upper bound 
on $\Phi_{K\wr H,\rho_\alpha}$, consider the measures 
$$\mu_{H,\alpha}(h)\asymp (1+|h|)^{-\alpha-d} \mbox{ on } H \mbox{ and } 
\mu_{K,\alpha}(k)\asymp \sum_1^\infty 4^{-\alpha k}\mathbf u_{4^k} \mbox{ on }K.
$$
They satisfy $W(\rho_\alpha,\mu_{H,\alpha})<\infty $ 
and $W(\rho_\alpha,\mu_{K,\alpha})<\infty$ and this immediately implies 
$W(\rho_\alpha,\mu)<\infty$ 
where $\mu=\frac{1}{2}(\mu_{H,\alpha}+\mu_{K,\alpha})$ is understood as 
a probability measure on $K\wr H$. 
By  Theorems \ref{th-LW}-\ref{th-A-pol-Lambda}, 
$\Lambda_{2,K,\mu_{K,\alpha}}(v)\simeq (\log (e+ v))^{-\alpha}$ and 
$\Lambda_{2,H,\mu_{H,\alpha}}(v)\simeq v^{-\alpha/d}$. Theorem \ref{th-WR}
Lemma \ref{lem-IsoNash} and Theorem \ref{th-Coulhon1} give the desired result.
\end{proof}

\begin{theo} \label{th-wr-Pol-rho}
Let $H$ be a group of polynomial volume growth of degree $d$. 
Let $\rho$ be a regularly varying function of index $2$ and set
$M(t)= s^2/\int_0^t\frac{sds}{\rho(s)}$.  Assume that 
$\theta(t)=\int_0^t\frac{sds}{\rho(s)}$ satisfies 
$\theta(t^a)\simeq \theta(t)$ for each $a>0$.
\begin{enumerate}
\item If $K\neq \{e_K\}$ is finite, we have
$$\widetilde{\Phi}_{K\wr H,\rho}(n)\simeq 
\exp\left(-\left(n \int_0^{n} \frac{sds}{\rho(s)}\right)^{d/(2+d)}\right).$$
\item If $K$ is not finite and has polynomial volume growth, we have
$$\widetilde{\Phi}_{K\wr H,\rho}(n)\simeq 
\exp\left(-\left(n (\log n)^{2/d} \int_0^{n} \frac{sds}{\rho(s)}\right)^{d/(2+d)}\right) .$$
\item If $K$ satisfies $\Phi_{K}(n)\gtrsim 
\exp(-n^{1/3})$, we have
$$\widetilde{\Phi}_{K\wr H,\rho}(n)\gtrsim
\exp\left(-\left(n (\log n)^{2/(d+1)} \int_0^{n} \frac{sds}{\rho(s)}
\right)^{(d+1)/(d+3)}\right).$$
\end{enumerate}
\end{theo}
\begin{proof}The lower bounds follow from Theorem \ref{th-rho-low}. 
The upper bounds follow from Proposition \ref{pro-wreathup}. 
Note that the upper bound is missing in the last case.  
We outline the upper-bound argument in case 2. Consider the measures 
$$\mu_{G,\rho}(g)\asymp \frac{1}{(1+|g|)^{2+d}\ell(1+|g|)}$$
for $G=H$ and $G=K$. By Proposition \ref{pro-relppi}, we have
$$\mathcal E_{G,\mathbf u_r} \le C \frac{r^2}{\theta (r)} \mathcal E_{\mu_{G,\rho}}$$
for $G=H,K$.  This allows us to verify the hypotheses of Proposition 
\ref{pro-wreathup} with  $H_1=H$, $H_2=K$,
$\zeta_{i,v_i(t)}=\mathbf u_{r(2v_i(t))}$, $r(v_i)\simeq v^{1/d_i}$ and 
$v_i(t)\simeq ( t \theta(t))^{d_i/2})$, where $d_1=d$ and $d_2$ is the 
degree of polynomial volume growth of $K$.  In the notation of Proposition 
\ref{pro-wreathup}, this gives $v(t)\simeq \exp( [t \theta(t)]^{d/2}\log t)$    
which translates into
$$\Lambda_{K\wr H,\mu}(v) \gtrsim 
\frac{(\log\log v)^{2/d}\theta(\log v)}{(\log v)^{2/d}}$$
where the measure $\mu$ on $K\wr H$ is given by 
$\mu=\frac{1}{2}(\mu_{H,\rho}+\mu_{K,\rho})$.
With this estimate in hand, Theorem \ref{th-Coulhon1} gives
$\mu^{(2n)}(e)\lesssim  \psi(n)$ where $\psi$ is given 
implicitly as a function of $t$ by
$$t= \int_1^{1/\psi} 
\frac{(\log v)^{2/d}}{[\log(e+\log v)]^{2/d}\theta(\log v)}.$$ 
A somewhat tedious computation shows that this equality gives 
$$t\simeq  \frac{[\log (1/\psi)]^{(2+d)/d}}{[\log \log (1/\psi)]^{2/d} 
\theta (\log (1/\psi))}$$  
or, equivalently,
$$\log (1/\psi) \simeq  \left( t (\log t)^{2/d} \theta (t)\right)^{d/(2+d)}.$$
Note that the assumed property that $\theta(t^a)\simeq \theta(t)$ for $a>0$
has been used repeatedly in these computations.
This gives the desired upper bound on $\mu^{(2n)}(e)$ and thus on 
$\Phi_{K\wr H,\rho}$ as well.  
\end{proof}

\begin{rem} In the third statement of Theorem \ref{th-wr-Pol-rho}, 
even if we assume in addition that $K$ has exponential volume growth
(in which case $\Phi_{K}(n)\simeq \exp(-n^{1/3})$), we would still 
not be able to state a matching upper bound.
The reason  is that we do not 
have at our disposal the appropriate pseudo-Poincar\'e inequality on $K$ 
(in the case when $K$ has polynomial volume growth, we used Proposition 
\ref{pro-relppi}).  However, consider the special case when 
$K=F\wr \mathbb Z$ with $F\neq \{e\}$ finite. This group has 
exponential volume growth and satisfies $\Phi_{K}(n)\simeq \exp(-n^{1/3})$.
Further, Proposition \ref{pro-wreathup} applied with $H_1=K$, $H_2=F$ provides 
us with a measure $\mu_{K,\rho}$ on $K$ (and accompanying measures $\zeta_{v}$) 
which is a good witness for $\widetilde{\Phi}_{K,\rho}$ and can be used 
to apply Proposition \ref{pro-wreathup} with $H_1=H$, $H_2=K=F\wr \mathbb Z$
$\mu_1=\mu_{H,\rho}$ as above and $\mu_2=\mu_{K,\rho}$ (the measure just 
obtained on $K=F\wr \mathbb Z$). After elementary but tedious computations,
Proposition \ref{pro-wreathup} implies that
the measure 
$\mu=\frac{1}{2}(\mu_1+\mu_2)$ 
on $K\wr H =(F\wr \mathbb  Z)\wr H$ satisfies
$$\mu^{(2n)}(e)\le 
\exp\left(-\left(n (\log n)^{2/(d+1)} \int_0^{n} \frac{sds}{\rho(s)}
\right)^{(d+1)/(d+3)}\right).$$
This shows that
$$\Phi_{(F\wr \mathbb Z)\wr H,\rho}(n) \simeq 
\exp\left(-\left(n (\log n)^{2/(d+1)} \int_0^{n} \frac{sds}{\rho(s)}
\right)^{(d+1)/(d+3)}\right).$$
In particular, the lower bound stated in Theorem \ref{th-wr-Pol-rho}
is sharp in this case. We conjecture that it is also sharp when 
$K$ is polycyclic of exponential volume growth.
\end{rem}

\subsection{Anisotropic measures on nilpotent groups} 

This section is concerned with special cases of the following problem 
raised and studied in \cite{SCZ-nil}.  Given a group $G$ generated
by a $k$-tuple $S=(s_1,\dots,s_k)$, study the behavior 
of the random walks driven by the measures
$$\mu_{S,a}(g)=\frac{1}{k}\sum_1^k \sum_{m\in \mathbb Z}
\frac{c_i}{(1+|m|)^{1+\alpha_i}}
\mathbf 1_{s_i^m}(g)$$ 
where $a=(\alpha_1,\dots,\alpha_k)\in (0,\infty)^k$ and 
$c_i=(\sum_\mathbb Z (1+|m|)^{-1-\alpha_i})^{-1}$ 
In words, to take a step according to $\mu_{S,a}$, 
pick one of the $k$ generators, say $s_i$, uniformly at random. Independently, pick an integer $m\in \mathbb Z$ according to the power law 
giving probability $c_i(1+|m|)^{-1-\alpha_i}$ to $m$. Then
multiply the present position (on the right) by $s_i^m$.  

It is of interest to investigate the behavior of $\mu_{S,a}^{(n)}(e)$
and to understand how, given the generating $k$-tuple $S$, 
this behavior depends on the choice of $a=(\alpha_1,\dots,\alpha_k)$.  
Here we use the results of \cite{SCZ-nil} and the 
technique of the present paper to prove the following result.

We consider two groups $H=H_1$ and $K=H_2$. The group $H$ is assumed to 
be nilpotent generated by the $S=(s_1,\dots ,s_{p})$. 
On this nilpotent group, we considered the measures $\mu_{H,S,a}$ 
with $a=(\alpha_1,\dots,\alpha_p)\in (0,2)^p$. The group $K$ will be either 
finite or nilpotent. If it is finite, we let $\mu_{K}$ be the uniform 
measure on $K$. If $K$ is nilpotent, generated by a given tuple 
$T=(t_1,\dots,t_q)$, we consider the measures $\mu_{K,T,b}$ with
$b=(\beta_1,\dots,\beta_q)\in (0,2)^q$. 

Next we consider the wreath product $G=K\wr H$. When $K$ is nilpotent,
the generating sets $S$ and $T$ (for $H$ and $K$, respectively) together 
produce a generating set 
$\Sigma=(\sigma_1,\dots, \sigma_k)$, $k=p+q$, of $K\wr H$ 
where $\sigma_1,\dots, \sigma_p$ corresponds to $s_1\dots, s_p$ and 
generates $H$ inside $K\wr H$ and the generators 
$\sigma_{p+1},\dots \sigma_{p+q}$ correspond to $t_1\dots, t_q$ and generate
the copy of $K$ in $K\wr H$ which seats at $e_H$. Similarly, set $c=(\gamma_1,\dots,\gamma_k)$ with $\gamma_i=\alpha_i, 1\le i\le p$ and $\gamma_{i}=
\beta_{i-p}$, $i=p+1,\dots, p+q=k$.
By elementary Dirichlet form comparison arguments, we know that the measures 
$$\mu=\frac{1}{2}(\mu_{H,S,a}+\mu_{K,T,b}) \mbox{ and }
\mu_{K\wr H, \Sigma, c} \mbox{ on } G=K\wr H$$
satisfy $\mu^{(2n)}(e)\simeq \mu_{\Sigma, c}^{(2n)}(e)$.
\begin{theo} Let $H,S,p$ and  $a\in (0,2)^p$  
be as above. Let $d(a)$ be the real given 
by {\em \cite[Theorem 1.8]{SCZ-nil}} and such that 
$\mu_{S,a}^{(n)}(e)\simeq n^{-d(a)}$.
\begin{enumerate}
\item  Assume that $K\neq \{e\}$ is finite. Then 
the measure $\mu=\frac{1}{2}(\mu_{S,a}+\mu_K)$ on $K\wr H$ satisfies
$$\mu^{(n)}(e)\simeq \exp\left(-n^{d(a)/(1+d(a))}\right).$$
\item Assume that $K$ is nilpotent (infinite)  
and $T$, $q$ and $b\in (0,2)^q$ are as 
described above. Then, on $G=K\wr H$ equipped with the generating set 
$\Sigma$ define  above, the measure $\mu_{\Sigma,c}$ satisfies
$$\mu_{\Sigma,c}^{(n)}(e)\simeq \exp\left(-n^{d(a)/(1+d(a))}(\log n)^{1/(1+d(a))}\right).$$
\end{enumerate}
\end{theo}
\begin{proof}This follows from Theorem \ref{th-WR} because 
\cite[Theorem 1.8]{SCZ-nil}
shows that the measures $\mu_1=\mu_{H,S,a}$ and $\mu_2=\mu_{K,T,b}$
satisfy  $\mu_{i}^{(n)}(e)\simeq n^{-d_i}$ where $d_1=d(a)$, $d_2=d_(b)$ 
are the real described in  \cite[Theorem 1.8]{SCZ-nil} (recall that these 
estimates are equivalent to $\Lambda_{2,\mu_i}(v) \simeq v^{-2/d_i}$).
\end{proof}

\subsection{Local time functionals}\label{sec-localtime}

Let $H$ be a group equipped with a symmetric  measure $\mu$.  Let $\ell(x,n)$
be the number of visits to $x$ up to time $n$. More precisely,
let $(X_n)$  denotes the trajectory of a random walk 
driven by $\mu$ on $H$ and set 
$$l(n,x)=\#\{0< k\leq n: X_k=x\}.$$ 
It is well-known that the behavior of the probability 
of return of the switch-walk-switch random walk on the lamplighter group
$(\mathbb Z/2\mathbb Z)\wr H$ is related to certain functionals of the local 
times $(\ell(x,n))_{x\in H}$.
More precisely and more generally, let $K$ be a finitely generated group 
(possibly finite). Let $\mu_K$ be a symmetric measure on $K$
satisfying $\mu_K(e_K)>0$. Let $q=\mu_K*\mu*\mu_K$ be the switch-walk-switch 
measure on $K\wr H$ (see, e.g., \cite{SCZ-dv1} for details. With this notation,
we have
\begin{equation*}
q^{(n)}((\boldsymbol{e}_K,h)) \asymp \mathbf{E}_\mu
\left( \prod\limits_{x\in H}\nu ^{(2l(n,x))}(e_K)\boldsymbol{1}_{\{X_{n}=h\}
}\right)
\end{equation*}
where $\mathbf E_\mu$ and $(X_n)$ refers to the random walk on $H$ 
driven by $\mu$.

Set%
\begin{equation*}
F_{K}(n):=-\log \nu ^{(2n)}(e_K)
\end{equation*}%
so that, for any $h\in H$, 
\begin{equation}  \label{trick}
q^{(n)}((\boldsymbol{e}_K,h))
\simeq \mathbf{E}\left( e^{
-\sum_{x\in H}F_{K}(l(n,x))}\boldsymbol{1}_{\{X_{n}=h\}}\right).
\end{equation}

Assume next that, for each $R>0$ there is a set $U_R\subset of H$ and 
$\kappa\ge 1$ such that
\begin{equation}\label{trick-U}
|U_R|\le R^\kappa  \mbox{ and }   
\mu^{(n)}( H\setminus U_R)\le Cn^\kappa(1+ R/n)^{-1/\kappa}.
\end{equation}
We note that the second condition follows easily from the tail 
condition
\begin{equation}\label{trick-U1}
\mu( H\setminus U_R)\le C(1+ R)^{-1/\kappa}
\end{equation}
when $U_R=\{h: N(h)>R\}$ where $N: H\ra [0,\infty)$ satisfies 
$N(h_1h_2)\le N(h_1)+N(h_2)$. Indeed, under such circumstances, we have
$$\mu^{(n)}( H\setminus U_R)\le n\mu(H\setminus U_{R/n})\le 
Cn(1+R/n)^{-1/\kappa}.$$
Writing
\begin{eqnarray*}
\lefteqn{\mathbf{E}_\mu\left( e^{
-\sum_{x\in H}F_{K}(l(n,x))}\right)
=}&&\\ &&
\mathbf{E}_\mu\left( \sum_{h\in U_R}e^{
-\sum_{x\in H}F_{K}(l(n,x))}\boldsymbol{1}_{\{X_{n}=h\}}\right)\\
&&+\mathbf{E}_\mu\left(\sum_{h\in H\setminus U_R} e^{
-\sum_{x\in H}F_{K}(l(n,x))}\boldsymbol{1}_{\{X_{n}=h\}}\right)\\
&\le& |U_R| q^{(2[n/2])}(e_{K\wr H}) + \mu^{(n)}(H\setminus U_R)
\end{eqnarray*}
shows that, under assumption (\ref{trick-U}) and assuming that
$q^{(n)}(e_{K\wr H})\simeq \exp(- \omega(n))$ with 
$\omega(n)$  regularly varying of index in $(0,1]$,  we can conclude that
\begin{equation}
\mathbf{E}_\mu\left( e^{
-\sum_{x\in H}F_{K}(l(n,x))}\right)\simeq \exp(-\omega(n))
\end{equation}
as well.

This technique and remarks, together with Theorems 
\ref{th-wr-Pol-alpha}-\ref{th-wr-Pol-rho}, suffice to prove 
the following results.

\begin{cor} Let $H$ be a group of polynomial volume growth of degree $d$. Let 
$\mu_\alpha(h)\asymp (1+|h|)^{-\alpha-d}$, $\alpha>0$. Let $R_n$ be the 
number of visited point up to time $n$. For any fixed $\kappa>0$,
\begin{itemize}
\item If $\alpha>2$ then $\mathbf{E}_{\mu_\alpha}\left( e^{
-\kappa R_n}\right)\simeq \exp(-n^{d/(2+d)})$.
\item If $\alpha=2$ then $\mathbf{E}_{\mu_2}\left( e^{
-\kappa R_n}\right)\simeq \exp(-(n \log n)^{d/(2+d)})$.
\item If $\alpha\in (0,2)$ then $\mathbf{E}_{\mu_\alpha}\left( e^{
-\kappa R_n}\right)\simeq \exp(-n^{d/(\alpha+d)})$. 
\end{itemize}
\end{cor}  
\begin{rem}Note that the second case, $\alpha=2$, may be new even in 
the case when $H=\mathbb Z$. It gives the behavior of 
$\mathbf E(e^{-\kappa R_n})$ for the walk on $\mathbb Z$ driven by 
the measure $\mu_2(z)=c(1+|z|)^{-2}$ for which there is no 
classical local limit theorem and to which the classical Donsker-Varadhan 
theorem does not apply.
\end{rem}
\begin{cor} Let $H$ be a group of polynomial volume growth of degree $d$. Let 
Let $R_n$ be the 
number of visited point up to time $n$. Consider the
random walk driven by $\mu_{S,a}$ with $a=(\alpha_1,\dots,\alpha_k)\in (0,2)^k$.
Let $d(a)$ be the real given 
by {\em \cite[Theorem 1.8]{SCZ-nil}} and such that 
$\mu_{S,a}^{(n)}(e)\simeq n^{-d(a)}$.
For any fixed $\kappa>0$, we have
$$\mathbf E_{\mu_{S,a}}(e^{-\kappa R_n}) \simeq \exp\left(-n^{d(a)/(1+d(a))}\right).$$
\end{cor} 

Given a measure $\mu$ such as $\mu_\alpha$ or $\mu_{S,a}$ on $H$, and fixed $\kappa>0$ and $\gamma\in(0,1)$,  we can determine the behavior of
$$n\mapsto \mathbf E (e^{-\kappa\sum_H\ell(n,h)^\gamma}).$$
Indeed, it suffices consider the wreath product $\mathbb Z\wr H$ with a 
measure $\phi$ on $\mathbb Z$ such that $\nu^{(2n)}(0)\simeq \exp(-n^\gamma)$.  
The choice
$\phi(x)\asymp (1+|x|)^{-1} [1+\log (1+|x|)]^{-1/\gamma}$ fulfills 
these requirements (see \cite{SCZ-pollog}).

\begin{cor} Let $H$ be a group of polynomial volume growth of degree $d$. 
Fix $\kappa>0$ and $\gamma\in (0,1)$
Let $\ell(n,x)$ be the number of visits to $x\in H$ up to time $n$.
Let $\mu_\alpha(h)\asymp (1+|h|)^{-\alpha-d}$, $\alpha\in (0,2)$,
or $\mu=\mu_{S,a}$ with $a=(\alpha_1,\dots,\alpha_k)\in (0,2)^k$.
In the second case, let $d(a)$ be the real given 
by {\em \cite[Theorem 1.8]{SCZ-nil}} and such that 
$\mu_{S,a}^{(n)}(e)\simeq n^{-d(a)}$ and set $\frac{d}{\alpha}=d(a)$. 
We have
$$\mathbf E_{\mu}(e^{-\kappa\sum_H\ell(n,h)^\gamma})\simeq 
\exp\left(-n^{(\alpha\gamma+d(1-\gamma))/(\alpha+d(1-\gamma))}\right)
.$$
\end{cor} 

\begin{rem} If $\mu(h)\asymp (1+|h|)^{-2-d}$ or $\mu=\mu_{S,a}$ with 
$a=(2,\dots,2)$, one gets
$$\mathbf E_{\mu}(e^{-\kappa\sum_H\ell(n,h)^\gamma})\simeq 
\exp\left(-n^{(2\gamma+d(1-\gamma))/(2+d(1-\gamma))}[\log n]^{d(1-\gamma)/(2+d(1-\gamma))}\right)
.$$
\end{rem}

\appendix


\section{ Appendix: Radial power laws on groups} \setcounter{equation}{0}
In this appendix, we compute the $L^p$-profiles  $\Lambda_{p,\phi}$
for  radial ``power law'' probability measures on certain groups.

\subsection{Norm-radial power laws}
Let $G$ be a countable group. Let $N: G\ra [0,\infty)$ be such that 
$N(e)=0$, $N(x^{-1})=N(x)$ and $N(xy)\le C_N(N(x)+N(y))$. Set 
$$V_N(r)=|\{x\in G:N(x)\le r\}|, \;\;B_N(m)=\{x\in G:N(x)\le m\}$$ and
$\mathbf \nu_m=V_N(m)^{-1}\mathbf 1_{B(m)}$. 
For $\alpha>0$, set $$\phi_\alpha=c_\alpha
\sum_1^\infty 4^{-\alpha k}\mathbf \nu_{4^k},\;\;c_\alpha= (4^\alpha-1)/(4^\alpha-2).$$ 

It is obvious that
\begin{equation}\label{pseudoPalpha}
\forall\,r=4^k,\;\;
\sum_{x,y}|f(xy)-f(x)|^p\nu_r(y) 
\le c_\alpha r^{\alpha } \sum_{x,y}|f(x)-f(xy)|^p\phi_\alpha(y).
\end{equation}
Let $W_N$ be the inverse function of the modified volume function defined 
by $\mathfrak V_N(r)= V_N(4^k)$ if $4^k\le r<4^{k+1}$, i.e., 
$\mathfrak W_N(t)=\inf\{s: \mathfrak V_N(s)>t\}$. 
Note that $\mathfrak V_N\simeq V_N$. 

\begin{pro} \label{pro-JLalpha}
Referring to the setup introduced above, for any $\alpha>0$ and 
$p\in [1,\infty)$, we have
$$\Lambda_{p,\phi_\alpha}(v)\ge \frac{1}{c_\alpha 8^p 
\mathfrak W_N(2^pv)^{\alpha}}.$$ 
\end{pro} 
\begin{proof}
The argument is well-known and is
given here for convenience of the reader. See also \cite{CSCiso,PSCnote}.
Consider a function $f\ge 0$ with $|\mbox{support}(f)|\le v$. 
For any $\lambda>0$, write
$$|\{ f\ge \lambda\}|\le |\{ |f-f*\nu_r|\ge \lambda/2\}|+
|\{ |f*\nu_r|\ge \lambda/2\}|$$
and note that $\|f*\nu_r\|_\infty\le V_N(r)^{-1/p}\|f\|_p.$

Recall the notation $f_k=(f-2^k)^+\wedge 2^k$ and observe that 
$\|f_k\|_p\le 2^{k} v^{1/p} $. It follows that 
$\|f_k*\nu_r\|_\infty\le  2^{k} V_N(r)^{-1/p} v^{1/p}.$     
Pick $r$ so that  $V_N(r) > 2^p v$ and pick $\lambda= 2^{k}$. We have
$$|\{f_k\ge 2^{k}\}|\le |\{ |f-f*\nu_r|\ge \lambda/2\}|\le 
 2^{-p(k-1)} r^{\alpha} \mathcal E_{p,\phi_\alpha}(f_k).$$
Recall that $|\{f\ge 2^{k+1}\}|= |\{f\ge 2^k\}|$ and write
\begin{eqnarray*}
\|f\|_p^p &\le & 8^p\sum _k 2^{p(k-1)} |\{f\ge 2^{k+1}\}|\\
&\le & c_\alpha 8^p r^\alpha
\sum_k \mathcal E_{p,\phi_\alpha}(f_k)\le
c_\alpha 8^p r^\alpha E_{p,\phi_\alpha}(f)
\end{eqnarray*}
where, for the last step, we have used (\ref{Hbcls}).
Given the choice of $r$ as a function of $v$, this gives
$$\Lambda_{p,\phi_\alpha}(v) \ge c_\alpha ^{-1}8^{-p} \mathfrak W_N(2^p v)^{-\alpha}$$
which is the desired inequality.
\end{proof}

In the case when $N=|\cdot|_N$ is the word-length associated with a finite 
symmetric generating set $S$, write $\mathcal W$ for the inverse function 
of the volume growth $V=V_S$.  Proposition \ref{pro-JLalpha} gives
$$
\Lambda_{p,\phi_\alpha} \gtrsim  \mathcal W^{-\alpha},\;\; p\ge 1,\;\;\alpha >0..$$
However, these inequalities compete with those deduced in a similar way from 
\begin{equation}\label{peudoP}
\forall\,r>0,\;\;
\|f-f_r\|^p_p \le C(p,S,\alpha)\; r^{p} \sum_{x,y}|f(x)-f(xy)|^p\phi_\alpha(y).
\end{equation} 
This inequality is an immediate consequence of the well-known 
pseudo-Poinca\'e inequality
$$
\forall\,r>0,\;\;
\|f-f_r\|^p_p \le C(p,S)\; r^{p} \sum_{x,y}|f(x)-f(xy)|^p\mathbf u (y)
$$
which follows from the definition of the word length
and a simple telescoping sum argument. See, e.g., \cite{CSCiso,PSCnote}. 

It follows that we have 
$$\Lambda_{p,\phi_\alpha} \gtrsim \left\{\begin{array}{ll}
\mathcal W^{-\alpha} & \mbox{ if } 
\alpha\in (0,p]\\
\mathcal W^{-p} &\mbox{ if } \alpha>p .\end{array}\right.
$$
In fact, because of the Dirichlet form comparison 
$\mathcal E_{\phi_\alpha}\simeq \mathcal E_{\mathbf u}$ 
which holds for $\alpha>2$ (see, e.g., \cite{PSCstab}), we must have 
$\Lambda_{\phi_\alpha}\simeq \Lambda_G$ for $\alpha>2$. Similarly, 
for $\alpha>p$,
we have 
$$\forall f,\;\;\sum_{x,y}|f(xy)-f(x)|^p\phi_\alpha(y) 
\asymp \sum_{x,y}|f(xy)-f(x)|^p\mathbf u (y)$$
and thus $\Lambda_{p,\phi_\alpha}\simeq \Lambda_{p,G}$.
In the case $p=1$, this implies that $J_{\phi_\alpha}\simeq J_G$ 
for all $\alpha>1$. This 
discussion is captured in the following result.
\begin{theo} \label{th-LW}
Let $G$ be a finitely generated group equipped with a
finite symmetric generating set and associated word-length. Set 
$\phi_\alpha=\sum_1^\infty 4^{-k \alpha}\mathbf u_{4^k}$ where $\mathbf u_r$ is the
uniform measure on the ball of radius $r$ in $G$. Let $W$ be the inverse 
function of the volume growth function of $G$.
\begin{itemize}
\item For $1\le p<\alpha<\infty$,  $\Lambda_{p,\phi_\alpha}\simeq \Lambda_{p,G}$. 
\item For $\alpha\in (0,p)$,  
we always have
$\mathcal W^{-\alpha}\lesssim \Lambda_{p,\phi_\alpha}\lesssim 
\Lambda_{p,G}^{\alpha/p}.$
\item 
If for a given $p\in [1,\infty)$ we have $\Lambda_{p,G}\simeq \mathcal W^{-p}$ 
then 
$$ \forall\, \alpha\in (0,p),\;\;\Lambda_{p,\phi_\alpha}\simeq \mathcal 
W^{-\alpha}.$$ 
\end{itemize}
\end{theo}
Note that the case $\alpha=p$ is excluded from this statement.
Note also that the wreath product construction provides many examples of groups 
for which  $\Lambda_{p,G}\not\simeq \mathcal W^{-p}$.
\begin{proof} The case when $\alpha>p$ is explained above as well as
the lower bounds when $\alpha\in (0,p]$. The 
upper bound for $\alpha \in (0,p)$ follows from Theorem 
\ref{pro-compphiG}.
\end{proof}

\begin{exa} Polycyclic groups satisfy $\Lambda_{p,G}\simeq 
\mathcal W^{-p}$ for each  
$p\in [1,\infty)$. 
The lower bound follows from the argument of \cite{CSCiso} as explained above.
The upper bound is best derived from the existence of adapted F\o lner couples,
a technique developed and explained in \cite{CGP}.  Other groups for which
$\Lambda_{p,G}\simeq \mathcal W^{-p}$ include the Baumslag solitar groups 
$\mathbf{BS}(1,m)=\langle a,b: aba^{-1}=b^m\rangle$ and the lamplighter groups
$F\wr \mathbb Z $ with $F$ finite.  Romain Tessera \cite[Theorem 4]{Tessera}
describes a large class of groups of exponential volume growth 
(Geometrically Elementary Solvable or GES groups) which satisfy 
$\Lambda_{p,G}\simeq \mathcal W^{-p}$. 
Note that what Tessera denotes by $j_{p,G}$ is
$1/\Lambda_{p,G}^{1/p}$. 
\end{exa}

\begin{rem} Recall the two sided Cheeger inequality (\ref{Cheegerpq}), i.e.,
$$c(p,q)\Lambda^{q/p}_{p,\phi}\le \Lambda_{q,\phi} \le C(p,q)\Lambda_{p,\phi},\;\;1\le p\le q<\infty.$$
Let $G$ be a group such that $\Lambda_{p,G}\simeq \mathcal 
W^{-p}$, $p\ge 1$ and fix 
$\alpha\in (0,\infty)$.  By Theorem \ref{th-LW}, if $p\in [1,\alpha)$,
$\Lambda_{p,\phi_\alpha}\simeq \mathcal W^{-p}$ but if $p>\alpha$, 
$\Lambda_{p,\phi_\alpha}\simeq \mathcal W^{-\alpha}$. 
In particular, if $p,q>\alpha$,
then $\Lambda_{p,\phi_\alpha}\simeq \Lambda_{q,\phi_\alpha}$ but, if $p,q\in [1,\alpha)$ then $\Lambda^{q/p}_{p,G}\simeq \Lambda_{p,G}$. In the case 
$1\le p<\alpha <q<\infty$, neither of the two sides of the Cheeger inequality 
is optimal. 
\end{rem}

\subsection{Word-length power laws on group with polynomial volume growth}
We now focus on
the case when $N(x)=|x|_S$ is the word-length of $x$ with respect 
to a finite symmetric generating set $S$ on a group of polynomial 
volume growth. Dropping the reference to the set $S$, we set
$V(r)=|\{x: |x|\le r\}|$ and assume that $V(r)\simeq r^{D}$, i.e., 
we assume that the group $G$ has polynomial volume growth of degree $D$. 
In this case, we can use a more refined version of the 
measure $\phi_\alpha$ by setting
$$\phi_\alpha=c_\alpha\sum_1^\infty k^{-\alpha-1} \mathbf u_{k}, \;\;
c_\alpha^{-1}=\sum_1^\infty k^{-1-\alpha}.$$ 
It is easy to use an Abel summation argument to check that 
$$\forall x\in G,\;\;\phi_\alpha(x)\asymp (1+|x|)^{-\alpha-D} .$$ 
(the same  holds true for the measure  
$c'_\alpha\sum_1^\infty 4^{-\alpha k} \mathbf u_{4^k}$). 
\begin{pro} \label{pro-relppi}
Let $G$ be a group with polynomial volume growth.
Then, for each $p\ge 1$ and $r>s\ge 1$, we have
$$\sum_{x,y}|f(xy)-f(x)|^p\mathbf u_r(y) \le C(G,p) (r/s)^p
\sum_{x,y}|f(xy)-f(x)|^p\mathbf u_s(y).$$
\end{pro}
\begin{proof} For any $1\le s\le s_0$, this follows from the usual
Dirichlet argument using paths. So, we can assume $s>3$ and let $s'$ be the 
largest integer smaller than $s/3$.   For any $y\in G$, write
$y=y_0y_1\dots y_k$ with $y_0=e$, $|y_i|\le s'$, $1\le i\le k$, and
$k\le  9|y|/s$.  For any finite supported function$f$,
$\xi_1,\dots, \xi_{k-1}\in G$ and $|y|\le r$,  
we have
$$|f(xy)-f(x)|^p\le  9 (r/s)^{p-1}\sum_1^k|f(z_0\dots z_i)-f(xz_0\cdots z_{i-1})|^p$$
where $z_i= \xi_{i-1}^{-1}y_i\xi_i$ with $\xi_0=\xi_k=e$. 
Summing over $x\in G$ gives
$$\sum_x|f(xy)-f(x)|^p\le  9 (r/s)^{p-1}\sum_1^k\sum_x|f(x \xi_{i-1}^{-1}y_i\xi_i)-f(x)|^p.$$
We now average this inequality over  
$$(\xi_0,\xi_1,\dots,\xi_{k-1},\xi_k)\in \{e\}\times B(s')\times \dots \times B(s')\times \{e\}$$
This gives
\begin{eqnarray*} \sum_x|f(xy)-f(x)|^p &\le & 9(r/s)^{p-1}\left( 
V(s')^{-1}\sum_x\sum_{|\xi|\le s'}|f(x y_1\xi)-f(x)|^p\right.\\
&& +V(s')^{-2}\sum_{i=2}^{k-1}\sum_x\sum_{\xi,\zeta \in B(s')}|f(x \xi^{-1}y_i\zeta)-f(x)|^p\\
&& \left. + V(s')^{-1}\sum_x\sum_{|\xi|\le s'}|f(x \xi^{-1}y_{k})-f(x)|^p 
\right).
\end{eqnarray*}
Obviously, we have
$$
\sum_x\sum_{|\xi|\le s'}|f(x y_1\xi)-f(x)|^p
\le V(s)\sum_{x,z}|f(x z)-f(x)|^p\mu_s(z)$$
and
$$\sum_x\sum_{|\xi|\le s'}|f(x \xi^{-1}y_k)-f(x)|^p
\le V(s)\sum_{x,z}|f(x z)-f(x)|^p\mu_s(z).$$
Similarly, we have
\begin{eqnarray*}
\sum_x\sum_{\xi,\zeta \in B(s')}|f(x \xi^{-1}y_i\zeta)-f(x)|^p
&\le & V(s)\sum_{\xi\in B(s')}\sum_{x,z}
|f(x z)-f(x)|^p\mathbf u_s(z)\\
&\le &
V(s)V(s')\sum_{x,z}
|f(x z)-f(x)|^p\mathbf u_s(z).
\end{eqnarray*}
Since $k\le 9(r/s)$, putting these inequalities together yields
$$\sum_x|f(xy)-f(x)|^p \le  9[V(s)/V(s')] (r/s)^{p}
\sum_x|f(xz)-f(x)|^p\mathbf u_s(z).$$
Averaging over $y\in B(r)$ gives the desired inequality.
\end{proof}
\begin{cor} \label{cor-ppi}
On group with polynomial volume growth, for $p\ge 1$ and 
$\alpha>0$  there exists a constant $C(G,p,\alpha)$ such that
$$\|f-f*\mathbf u_r\|_p^p\le C(G,p,\alpha) Q_{p,\alpha}(r) 
\sum_{x,z}|f(xy)-f(x)|^p\phi_\alpha(z)$$
where 
$$Q_{p,\alpha}(r)= \left\{\begin{array}{ll} r^p &\mbox{ if } \alpha>p,\\
r^p/\log(e+ r) &\mbox{ if } \alpha=p\\
r^{\alpha } & \mbox{ if } \alpha<p.
\end{array}\right.$$
\end{cor}
\begin{proof}Only the case $\alpha=p$ needs a proof. This case follows 
immediately from the previous proposition and the definition of $\phi_\alpha$.
\end{proof}
\begin{rem} Fix a continuous increasing function $\ell$
such that $\ell(2t)\le C\ell(t)$ and
$\int^{\infty}\frac{ds}{s\ell(s)}<\infty$. Let 
$\phi$ be a symmetric probability measure on the group $G$ (which we assume 
to have polynomial volume growth of degree $d$) and such that
$$\phi \asymp \sum_1^\infty 
\frac{1}{k\ell(k)}\mathbf u_k.$$ 
Proposition \ref{pro-relppi} immediately gives,
$$\mathcal E_{\mathbf u_r}\leq C(G,p,\ell) r^p 
\left(\int_r^{\infty}\frac{s^{p-1}}{\ell(s)}ds\right)^{-1}\mathcal E_\phi.$$ 
This covers the different cases of Corollary \ref{cor-ppi}. When 
$\ell $ is slowly varying, this estimate is often not sharp and  
a sharp version is provided in \cite{SCZ-pollog}.
 \end{rem}

\begin{theo}\label{th-A-pol-Lambda}
On a group with polynomial volume growth of degree $D$
and for any $1\le p <\infty$, we have
$$ \Lambda_{p,\phi_\alpha}(v)\simeq
\left\{\begin{array}{ll} v^{-p/D} &\mbox{ if } \alpha>p,\\
v^{-p/D}\log(e+ v) &\mbox{ if } \alpha=p,\\
v^{-\alpha /D} & \mbox{ if } 0<\alpha<p,\end{array}\right.$$
\end{theo}
\begin{proof}The lower bounds on $\Lambda_{p,\phi_\alpha}$
follow easily from the previous corollary and the argument
in \cite{CSCiso} as explained in the previous section. 
The upper bounds on $\Lambda_{p,\phi_\alpha}$ follows from Theorem
\ref{pro-compphiG}. For instance,
in the case $\alpha=1$, 
Theorem\ref{pro-compphiG} gives 
$$\Lambda_{1,\phi_1}(v)\lesssim (\log (e+ s))\Lambda_{1,G}(v) + \frac{1}{s}.$$
This is optimized by the choice  $s=1/\Lambda_{1,G}(v)$ and we know that 
$\Lambda_{1,G}(v)\simeq (v)^{-1/D}$. Hence
$\Lambda_{1,\phi_1}(v)\lesssim v^{-1/D}\log (e+ v)$.
\end{proof}

\bibliographystyle{amsplain}

\providecommand{\bysame}{\leavevmode\hbox to3em{\hrulefill}\thinspace}
\providecommand{\MR}{\relax\ifhmode\unskip\space\fi MR }
\providecommand{\MRhref}[2]{%
  \href{http://www.ams.org/mathscinet-getitem?mr=#1}{#2}
}
\providecommand{\href}[2]{#2}

\end{document}